\def\ifSIAM{\iffalse}

\ifSIAM
	\documentclass[a4paper,reqno]{siamltex}
	\input{preamble_siam}
\else
	\documentclass[a4paper,reqno]{amsart}
\long\def\short#1#2{#2}

\usepackage{xcolor}
\definecolor{darkorange}{rgb}{0.6 0.25 0}

\usepackage[ulem=normalem,draft]{changes}
\definechangesauthor[Roland Herzog]{RH}{red}
\definechangesauthor[Christian Meyer]{CM}{blue}
\definechangesauthor[Gerd Wachsmuth]{GW}{darkorange}
\setauthormarkup[right]{}

\usepackage[utf8]{inputenc}      
\usepackage[T1]{fontenc}         

\usepackage{amsfonts}            

\usepackage{natbib}              
\newcommand\doi[1]{\href{http://dx.doi.org/#1}{doi: \nolinkurl{#1}}}

\usepackage{booktabs}            
\usepackage{titlesec}            

\usepackage{xr-hyper}
\definecolor{darkgreen}{rgb}{0,0.5,0}
\usepackage[colorlinks=true,urlcolor=red,citecolor=blue,filecolor=darkgreen]{hyperref}
\usepackage{autoref_helper}

\renewcommand\subsectionautorefname{Section}

\titleformat{\section}{\normalfont\Large\bfseries}{\thesection}{1em}{}
\titleformat{\subsection}[runin]{\normalfont\bfseries}{\thesubsection.}{0.5em}{}[.]
\titleformat{\subsubsection}[runin]{\normalfont\bfseries}{\thesubsubsection.}{0.5em}{}[.]
\titleformat{\paragraph}[runin]{\normalfont\bfseries}{}{0.5em}{}[.]

\numberwithin{table}{section}    
\numberwithin{figure}{section}   
\numberwithin{equation}{section} 

\renewcommand\theenumi{(\arabic{enumi})}
\renewcommand\labelenumi\theenumi

\usepackage{LaTeX_boldsymbols}
\usepackage{LaTeX_mathcals}
\usepackage{LaTeX_others}
\usepackage{LaTeX_sets}
\usepackage{LaTeX_vectors}
\usepackage{Plasticity}

\renewcommand{\Uad}{U_\textup{ad}}
\newcommand{\hatUad}{\hat U_\textup{ad}}

\usepackage{mathrsfs}

\IfFileExists{framed_theorems.sty}{\usepackage[proof]{framed_theorems}}{\newcommand\maketheoremcolored[1]{}\newcommand\colorletwhiteprint[1]{}}

\newtheorem{theorem}{Theorem}[section]
\newtheorem{lemma}[theorem]{Lemma}
\newtheorem{corollary}[theorem]{Corollary}

\newtheorem{assumption}[theorem]{Assumption}
\newtheorem{remark}[theorem]{Remark}

\maketheoremcolored{theorem}
\maketheoremcolored{assumption}
\maketheoremcolored{lemma}
\maketheoremcolored{corollary}
\maketheoremcolored{proposition}
\maketheoremcolored{remark}
\maketheoremcolored{proof}

\definecolor{theoremframecolor}{rgb}{1.0,0.0,0.0}%
\colorletwhiteprint{theoremshadecolor}{theoremframecolor!10!white}%
\definecolor{assumptionframecolor}{rgb}{0.0,1.0,0.0}%
\colorletwhiteprint{assumptionshadecolor}{assumptionframecolor!10!white}%
\definecolor{lemmaframecolor}{rgb}{0.0,0.0,1.0}%
\colorletwhiteprint{lemmashadecolor}{lemmaframecolor!10!white}%
\definecolor{corollaryframecolor}{rgb}{0.0,0.0,1.0}%
\colorletwhiteprint{corollaryshadecolor}{corollaryframecolor!10!white}%
\definecolor{propositionframecolor}{rgb}{0.0,0.0,1.0}%
\colorletwhiteprint{propositionshadecolor}{propositionframecolor!10!white}%
\definecolor{remarkframecolor}{rgb}{1.0,0.5,0.0}%
\colorletwhiteprint{remarkshadecolor}{remarkframecolor!10!white}%
\definecolor{proofframecolor}{rgb}{0.5,0.5,0.5}%
\colorletwhiteprint{proofshadecolor}{proofframecolor!4!white}%

\fi

\begin{document}
\ifSIAM
\title{Optimal Control of Quasistatic Plasticity with Linear Kinematic Hardening\\Part I: Existence and Discretization in Time}
\author{Gerd Wachsmuth}
\else
\title[Optimal Control of Quasistatic Plasticity, \today]{Optimal Control of Quasistatic Plasticity with Linear Kinematic Hardening\\Part I: Existence and Discretization in Time}
\author[Gerd Wachsmuth, \protect\today]{Gerd Wachsmuth}
\address{Chemnitz University of Technology, Faculty of Mathematics, D--09107 Chemnitz, Germany}
\email{gerd.wachsmuth@mathematik.tu-chemnitz.de}
\urladdr{http://www.tu-chemnitz.de/mathematik/part\_dgl/people/wachsmuth}
\fi

\maketitle

\begin{abstract}
	In this paper we consider an optimal control problem governed by a time-dependent variational inequality arising in quasistatic plasticity with linear kinematic hardening.
	We address certain continuity properties of the forward operator, which imply the existence of an optimal control.
	Moreover, a discretization in time is derived and we show that every local minimizer of the continuous problem
	can be approximated by minimizers of modified, time-discrete problems.
\end{abstract}


\section{Introduction}
\label{sec:intro}
The optimization of elastoplastic systems is of significant importance for industrial deformation processes, e.g., for the control of the springback of deep-drawn metal sheets.
In this paper we consider an optimal control problem for the quasistatic problem of small-strain elastoplasticity with linear kinematic hardening.
The strong formulation of the forward system (in the stress based, so-called dual formulation) reads (cf.\ \cite[Chapters~2, 3]{HanReddy1999})
\begin{equation}
	\label{eq:Forward_problem_strong_form}
	\left.
		\begin{aligned}
			\C^{-1} \dot\bsigma - \bvarepsilon(\dot\bu) + \lambda \, (\bsigma^D + \bchi^D) & = \bnull & & \mrep{\text{in }}{\text{on }} (0,T)\times\Omega, \\
			\H^{-1} \dot\bchi \phantom{{}-\bvarepsilon(\dot\bu)} + \lambda \, (\bsigma^D + \bchi^D) & = \bnull & & \mrep{\text{in }}{\text{on }} (0,T)\times\Omega, \\
			\div \bsigma & = -\bf & & \mrep{\text{in }}{\text{on }} (0,T)\times\Omega, \\
			\begin{aligned}
				\smash{\text{with complem.\ conditions}} && \; 0 \le \lambda \; \perp \; \phi(\bSigma) \\
				\smash{\text{and boundary conditions}} && \;\bu 
			\end{aligned}
			& \hspace*{-0.5mm}
			\begin{aligned}
				& \le 0 \\
				& = \bnull 
			\end{aligned}
			& & 
			\mspace{-3mu}
			\begin{aligned}
				& \mrep{\text{in }}{\text{on }} (0,T)\times\Omega, \\
				& \text{on } (0,T)\times\Gamma_D, 
			\end{aligned} \\
			\bsigma \cdot \bn & = \bg & & \text{on } (0,T)\times\Gamma_N.
		\end{aligned}
	\quad \right\}
\end{equation}
The system \eqref{eq:Forward_problem_strong_form} is subject to the initial condition $(\bsigma(0),\bchi(0),\bu(0)) = \bnull$.
The state variables consist of the stress $\bsigma$ and back stress $\bchi$, whose values are symmetric matrices.
Both matrix functions are combined into the generalized stress $\bSigma = (\bsigma,\bchi)$.
The state variables also comprise the vector-valued displacement $\bu$ and the scalar-valued plastic multiplier $\lambda$ associated with the yield condition $\phi(\bSigma) \le 0$ of von Mises type, see \eqref{eq:Yield_function}.
The first two equations in \eqref{eq:Forward_problem_strong_form}, together with the complementarity conditions, represent the material law of quasistatic plasticity.
The tensors $\C^{-1}$ and $\H^{-1}$ are the inverses of the elasticity tensor and the hardening modulus, respectively, and $\bsigma^D$ denotes the deviatoric part of $\bsigma$, see \eqref{eq:deviator}.
As usual, $\bvarepsilon(\bv) = \big(\nabla \bv + (\nabla \bv)^\top\big)/2$ denotes the (linearized) strain, where $\nabla\bv$ is the gradient of $\bv$.
The third equation in \eqref{eq:Forward_problem_strong_form} is the equilibrium of forces.
By $\div \bsigma$ we denote the (row-wise) divergence of the matrix-valued function $\bsigma$.
The boundary conditions correspond to clamping on $\Gamma_D \subset \partial\Omega$ and the prescription of boundary loads $\bg$ on the remainder $\Gamma_N = \partial\Omega \setminus \Gamma_D$, whose outer unit normal vector is denoted by $\bn$.
We will see in \autoref{subsec:stress_based_formulation} that \eqref{eq:Forward_problem_strong_form} can be reformulated as a mixed, time-dependent and rate-independent variational inequality of the first kind.

The boundary loads $\bg$ act as control variables.
There would be no difficulty in admitting volume forces $\bf$ as additional control variables but for practical reasons and simplicity of the presentation, $\bf$ is assumed to be zero throughout.
The optimal control problem under consideration reads
\begin{equation}
	\label{eq:upper_level_in_intro}
	\tag{$\mathbf{P}$}
	\left.
		\begin{aligned}
			\text{Minimize} \quad & \psi(\bu) + \frac{\nu}{2} \norm{\bg}_{H^1(0,T;L^2(\Gamma_N;\R^d))}^2 \\
			\text{such that} \quad & (\bSigma, \bu, \lambda) \text{ solves the quasistatic plasticity problem~\eqref{eq:Forward_problem_strong_form}}, \\
			\text{and} \quad & \bg \in \Uad.
		\end{aligned}
	\quad\right\}
\end{equation}
Here, $\psi$ is a functional favoring certain displacements $\bu$.
A typical example would be to track the final deformation after unloading.
The set $\Uad$ realizes constraints on the control $\bg$.
The assumptions on $\psi$ and $\Uad$ are given in \autoref{asm:psi_lsc}, followed by a number of examples.

The aim of this paper and of the subsequent works \cite{Wachsmuth2011:2,Wachsmuth2011:4} is the derivation of first order necessary optimality conditions for \eqref{eq:upper_level_in_intro}.
Let us highlight the main contributions of this paper.
\begin{enumerate}
	\item
		We show that the derivative $(\dot\bSigma,\dot\bu)$ of the solution depends continuously
		on the derivative $\dot\bg$ of the right-hand side, see \autoref{cor:continuity_GG_in_H1}.
		This continuity result of the solution operator of (the weak formulation of) \eqref{eq:Forward_problem_strong_form}
		is novel compared to the stability results known from the literature.

	\item
		Since the solution map of the system \eqref{eq:Forward_problem_strong_form} is nonlinear,
		its weak continuity is not obvious.
		We prove this weak continuity in \autoref{thm:weak_continuity_solution_operator}.
		An immediate consequence is the existence of optimal controls of \eqref{eq:upper_level_in_intro}, see \autoref{thm:existence_continuous}.

	\item 
		We show the convergence of the solutions of a time-discrete variant of \eqref{eq:Forward_problem_strong_form} w.r.t.\ the time-step size.
		In particular, we prove a rate of convergence of the generalized stresses $\bSigma$ in $L^\infty(0,T)$ \emph{without} assuming additional regularity, see \autoref{lem:sigma_tau_to_sigma_in_L_infty}.
		Moreover, the convergence of the plastic multiplier $\lambda$ is shown in \autoref{thm:strong_convergence_lambda}.
		These findings improve the results known from the literature.

	\item 
		We study the approximation of local minima of \eqref{eq:upper_level_in_intro} by local minima of its time-discrete variant, see \autoref{subsec:approx_by_time_discrete}.
		The results of this section are interesting in their own right and they are also used
		in order to derive necessary optimality conditions of \eqref{eq:upper_level_in_intro},
		see \citeautoref{Wachsmuth2011:4:sec:weak_stationarity_quasistatic}.
		Moreover, these approximation properties are important for a numerical realization of the optimal control problem.
\end{enumerate}


There are only few references concerning the optimal control of rate-independent systems, see \cite{Rindler2008,Rindler2009,KocvaraOutrata2005,KocvaraKruzikOutrata2006,Brokate1987}.
In the case of an infinite dimensional state space (i.e.\ $W^{1,p}(0,T;X)$, $\dim X = \infty$),
there are no contributions providing optimality conditions for rate-independent optimal control problems.
\cite{Rindler2008,Rindler2009} study the optimal control of rate-independent evolution processes in a general setting.
The existence of an optimal control and the approximability by solutions of discretized problems is shown.
The system of quasistatic plasticity in its \emph{primal formulation} (see \cite[Section~7]{HanReddy1999}) is contained as a special case.
However, in contrast to our analysis, the boundedness of the control in $W^{1,\infty}(0,T;U)$ (cf.\ \cite[Assumption~(U)]{Rindler2008}) is required and no optimality conditions are proven.

A finite dimensional situation (i.e.\ with state space $W^{1,p}(0,T; \R^n)$) is considered in 
\cite{KocvaraOutrata2005,KocvaraKruzikOutrata2006}, who consider spatially discretized problems,
and
\cite{Brokate1987}, who deals with optimal control of an ODE involving a rate-independent part.

The \emph{static} (i.e.\ time-independent) version of the optimal control problem \eqref{eq:upper_level_in_intro} was considered in \cite{HerzogMeyerWachsmuth2009:2,HerzogMeyerWachsmuth2010:2}.
For locally optimal controls, optimality systems of B- and C-stationarity type were obtained.


Let us sketch the outline of this paper.
In the remainder of this section, we introduce the notation and fix the functional analytic framework.
Moreover, we state the weak formulation of \eqref{eq:Forward_problem_strong_form} and give some references concerning its analysis.
\autoref{sec:continuous} is devoted to the continuous optimal control problem \eqref{eq:upper_level_in_intro}.
We give a brief introduction to evolution variational inequalities (EVIs).
Restating \eqref{eq:Forward_problem_strong_form} as an EVI, we are able to prove some new continuity properties of its solution operator.
Moreover, we show the weak continuity of the control-to-state map of \eqref{eq:upper_level_in_intro}.

In \autoref{sec:time_discretization} we consider a discretization in time of the control problem \eqref{eq:upper_level_in_intro}.
We show two convergence properties of the solution of the time-discrete forward problems, similar to the strong and weak continuity of the solution map of \eqref{eq:Forward_problem_strong_form}.
Finally, these results are used in \autoref{subsec:approx_by_time_discrete}
to prove
that \emph{every} local minimizer $\bg$ of the control problem \eqref{eq:upper_level_in_intro} can be approximated
by local minimizers of a slightly modified time-discrete problem \hyperref[eq:time-discrete_ulp_mod]{\textup{($\mathbf{P}_{\bg}^{\tau}$)}}, see \autoref{thm:continuous_approximation_with_local_solutions_mod}.

This paper can be understood as a prerequisite for the analysis contained in \cite{Wachsmuth2011:2,Wachsmuth2011:4}.
Since the problems under consideration and their analysis are non-standard,
the results presented here are believed to be of independent interest.

We give a brief overview over \cite{Wachsmuth2011:2,Wachsmuth2011:4}.
We regularize the time-discrete forward problem and show the Fréchet differentiability of the associated solution map.
This result requires some subtle arguments.
The regularized time-discrete optimal control problems are differentiable
and consequently, optimality conditions can be derived in a straightforward way.
The passage to the limit in the regularization parameter $\varepsilon$ leads to an optimality system of C-stationary type for the time-discrete problem.
Finally, we pass to the limit with respect to the discretization in time.
This part also requires new convergence arguments.
Due to the weak mode of convergence of the adjoint variables, the sign condition for the multipliers is lost in the limit and we finally obtain a system of weak stationarity for the optimal control of \eqref{eq:Forward_problem_strong_form}.

\subsection{Notation and assumptions}
Our notation follows \cite{HanReddy1999} and \cite{HerzogMeyerWachsmuth2009:2}.

\subsubsection*{Function spaces}
Let $\Omega \subset \R^d$ be a bounded Lipschitz domain with boundary $\Gamma = \partial\Omega$ in dimension $d = 3$.
The boundary consists of two disjoint parts $\Gamma_N$ and $\Gamma_D$.   
We point out that the presented analysis is not restricted to the case $d = 3$, but for 
reasons of physical interpretation we focus on the three dimensional case.
In dimension $d = 2$, the interpretation of \eqref{eq:Forward_problem_strong_form} has to be slightly modified,
depending on whether one considers the plane strain or plane stress formulation.

We denote by $\S := \R^{d \times d}_\textup{sym}$ the space of symmetric $d$-by-$d$ matrices, endowed with the inner product $\bsigma \dprod \btau = \sum_{i,j=1}^d \sigma_{ij} \tau_{ij}$, and we define
\begin{align*}
	V &= H^1_D(\Omega;\R^d) = \{\bu \in H^1(\Omega;\R^d): \bu = \bnull \text{ on } \Gamma_D \},
	&
	S &= L^2(\Omega;\S)
\end{align*}
as the spaces for the displacement $\bu$, stress $\bsigma$, and back stress $\bchi$, respectively.
The control $\bg$ belongs to the space
\begin{equation*}
	U = L^2(\Gamma_N; \R^d).
\end{equation*}
The control operator $E : U \to V'$, $\bg \mapsto \ell$, which maps boundary forces (i.e.\ controls) $\bg \in U$ to functionals (i.e.\ right-hand sides of the weak formulation of \eqref{eq:Forward_problem_strong_form}, see \autoref{subsec:stress_based_formulation}) $\ell \in V'$ is given by
\begin{equation}
	\label{eq:def_E}
	\dual{\bv}{E\bg}_{V,V'} := -\int_{\Gamma_N} \bv \cdot \bg \, \d s
	\quad
	\text{for all } \bv \in V.
\end{equation}
Hence, $E = -\tau_N^\star$, where $\tau_N$ is the trace operator from $V$ to $U = L^2(\Gamma_N; \R^d)$.
Clearly, $E: U \to V'$ is compact.

Starting with \autoref{sec:evolution_vi}, we will omit the indices on the duality bracket $\dual{\cdot}{\cdot}$.
We denote by $\dual{\cdot}{\cdot}$ the dual pairing between $V$ and its dual $V'$, or the scalar products in $S$ or $S^2$, respectively.
This will simplify the notation and cause no ambiguities.

For a Banach space $X$ and $p \in [1,\infty]$, we define the Bochner-Lebesgue space
\begin{equation*}
	L^p(0,T;X) = \{u : [0,T] \to X, \; u \text{ is Bochner measurable and $p$-integrable}\}.
\end{equation*}
In the case $p = \infty$ one has to replace $p$-integrability by essential boundedness.
The norm in $L^p(0,T;X)$ is given by
\begin{equation*}
	\norm{u}_{L^p(0,T;X)} = \bignorm{ \norm{u(\cdot)}_X }_{L^p(0,T)}.
\end{equation*}
By $W^{1,p}(0,T;X)$ we denote the Bochner-Sobolev space consisting of functions $u \in L^p(0,T;X)$ which possess a weak derivative $\dot u \in L^p(0,T;X)$.
Two equivalent norms on $W^{1,p}(0,T;X)$ are given by
\begin{equation}
	\label{eq:norms_in_w1p}
	\big(
		\norm{u}_{L^p(0,T;X)}^p
		+
		\norm{\dot u}_{L^p(0,T;X)}^p
	\big)^{1/p}
	\quad\text{and}\quad
	\big(
		\norm{u(0)}_X^p
		+
		\norm{\dot u}_{L^p(0,T;X)}^p
	\big)^{1/p},
\end{equation}
where the extension to the case $p = \infty$ is clear.
We use $H^1(0,T;X) = W^{1,2}(0,T;X)$.
Moreover, we define the space of functions in $H^1(0,T;X)$ vanishing at $t = 0$
\begin{equation}
	\label{eq:h01}
	H_{\{0\}}^1(0,T;X) = \{ u \in H^1(0,T;X) : u(0) = 0\}.
\end{equation}
Details on Bochner-Lebesgue and Bochner-Sobolev spaces can be found in \cite{Yosida1965}, \cite{GajewskiGroegerZacharias1974}, \cite{DiestelUhl1977}, or \cite{Ruzicka2004}.

\subsubsection*{Yield function and admissible stresses}
We restrict our discussion to the von Mises yield function.
In the context of linear kinematic hardening, it reads
\begin{equation}
	\label{eq:Yield_function}
	\phi(\bSigma) = \big( \abs{\bsigma^D + \bchi^D}^2 - \tilde \sigma_0^2 \big) / 2
\end{equation}
for $\bSigma = (\bsigma,\bchi) \in S^2$, where $\abs{\cdot}$ denotes the pointwise Frobenius norm of matrices and 
\begin{equation}
	\label{eq:deviator}
	\bsigma^D = \bsigma - \frac{1}{d} \, (\trace \bsigma) \, \bI
\end{equation}
is the deviatoric part of $\bsigma$.
The yield function gives rise to the set of admissible generalized stresses
\begin{equation*}
	\KK = \{ \bSigma \in S^2: \phi(\bSigma) \le 0 \quad \text{a.e.\ in } \Omega\}.
\end{equation*}
Let us mention that the structure of the yield function $\phi$ given in \eqref{eq:Yield_function}
implies the \emph{shift invariance}
\begin{equation}
	\label{eq:shift-invariance}
	\bSigma \in \KK
	\quad\Leftrightarrow\quad
	\bSigma + (\btau, -\btau) \in \KK
	\quad\text{for all }\btau \in S.
\end{equation}
This property is exploited quite often in the analysis.

Due to the structure of the yield function $\phi$, $\bsigma^D + \bchi^D$ appears frequently and we abbreviate it and its adjoint by
\begin{equation}
	\label{eq:DD}
	\DD \bSigma = \bsigma^D + \bchi^D
	\quad\text{and}\quad
	\DD^\star\bsigma =
	\begin{pmatrix} \bsigma^D\\\bsigma^D
	\end{pmatrix}
\end{equation}
for matrices $\bSigma \in \S^2$ as well as for functions $\bSigma \in S^2$ and $\bSigma \in L^p(0,T;S^2)$. 
When considered as an operator in function space, $\DD$ maps $S^2$ and $L^p(0,T;S^2)$ continuously into $S$ and $L^p(0,T;S)$, respectively.
For later reference, we also remark that 
\begin{equation*}
	\DD^\star \DD \bSigma = 
	\begin{pmatrix}
		\bsigma^D + \bchi^D \\
		\bsigma^D + \bchi^D \\
	\end{pmatrix}
	\quad \text{and} \quad (\DD^\star \DD)^2 = 2\,\DD^\star \DD
\end{equation*}
holds.
Due to the definition of the operator $\DD$, the constraint $\phi(\bSigma) \le 0$ can be formulated as $\norm{\DD\bSigma}_{L^\infty(\Omega;\S)} \le \tilde\sigma_0$.
Hence, we obtain
\begin{equation}
	\label{eq:bSigma_in_Linfty}
	\bSigma \in \KK
	\quad\Rightarrow\quad
	\DD\bSigma \in L^\infty(\Omega;\S).
\end{equation}

Here and in the sequel
we denote linear operators, e.g.\ $\DD : S^2 \to S$, and the induced Nemytzki operators, e.g.\ $\DD : H^1(0, T; S^2 ) \to H^1(0, T; S)$ and $\DD : L^2(0, T; S^2) \to L^2(0, T; S)$, with the same symbol.
This will cause no confusion, since the meaning will be clear from the context.

\subsubsection*{Operators}
The linear operators $A : S^2 \to S^2$ and $B : S^2 \to V'$
are defined as follows.
For $\bSigma = (\bsigma,\bchi) \in S^2$ and $\bT = (\btau,\bmu) \in S^2$, let $A\bSigma$ be defined through
\begin{equation}
	\label{eq:Definition_of_a}
	\dual{\bT}{A\bSigma}_{S^2} = \int_\Omega \btau \dprod \C^{-1} \bsigma \, \dx + \int_\Omega \bmu \dprod \H^{-1} \bchi \, \dx.
\end{equation}
The term $(1/2) \, \dual{A \bSigma}{\bSigma}_{S^2}$ corresponds to the energy associated with the stress state $\bSigma$.
Here $\C^{-1}(x)$ and $\H^{-1}(x)$ are linear maps from $\S$ to $\S$ (i.e., they are fourth order tensors) which may depend on the spatial variable $x$.
For $\bSigma = (\bsigma,\bchi) \in S^2$ and $\bv \in V$, let
\begin{equation}
	\label{eq:Definition_of_b}
	\dual{B\bSigma}{\bv}_{V',V} = - \int_\Omega \bsigma \dprod \bvarepsilon(\bv) \, \dx.
\end{equation}
We recall that $\bvarepsilon(\bv) = \big(\nabla \bv + (\nabla \bv)^\top\big)/2$ denotes the (linearized) strain tensor.

\subsubsection*{Standing assumptions}
Throughout the paper, we require
\begin{assumption}\hfill
	\label{asm:standing_assumptions}
	\begin{enumerate}
		\item The domain $\Omega \subset \R^d$, $d = 3$ is a bounded Lipschitz domain 
			in the sense of \cite[Chapter~1.2]{Grisvard1985}. The boundary of $\Omega$, denoted by $\Gamma$, 
			consists of two disjoint measurable parts $\Gamma_N$ and $\Gamma_D$ such that 
			$\Gamma = \Gamma_N \cup \Gamma_D$.
			While $\Gamma_N$ is a relatively open subset, $\Gamma_D$ is a relatively closed subset of $\Gamma$. 
			Furthermore $\Gamma_D$ is assumed to have positive measure.
			In addition, the set $\Omega \cup \Gamma_N$ is regular in the sense of
			Gröger, cf.~\cite{Groeger1989}. 
			A characterization of regular domains for the case $d \in \{2,3\}$ can be found in \cite[Section~5]{HallerDintelmannMeyerRehbergSchiela2009}.
			This class of domains covers a wide range of geometries. 

		\item The yield stress $\tilde \sigma_0$ is assumed to be a positive constant.
			It equals $\sqrt{2/3} \, \sigma_0$, where $\sigma_0$ is the uni-axial yield stress.

		\item $\C^{-1}$ and $\H^{-1}$ are elements of $L^\infty(\Omega;\LL(\S,\S))$, where $\LL(\S,\S)$ denotes the space of linear operators $\S \to \S$.
			Both $\C^{-1}(x)$ and $\H^{-1}(x)$ are assumed to be uniformly coercive.
			Moreover, we assume that $\C^{-1}$ and $\H^{-1}$ are symmetric, i.e., $\btau \dprod \C^{-1}(x) \, \bsigma = \bsigma \dprod \C^{-1}(x) \, \btau$ and a similar relation for $\H^{-1}$ holds for all $\bsigma, \btau \in \S$.
	\end{enumerate}
\end{assumption}

\hyperref[asm:standing_assumptions]{\autoref*{asm:standing_assumptions}~(1)} is not restrictive.
It enables us to apply the regularity results in \cite{HerzogMeyerWachsmuth2009:3} pertaining to systems of nonlinear elasticity.
The latter appear in the time-discrete forward problem and its regularizations.
Additional regularity leads to a norm gap, which is needed to prove the differentiability of the control-to-state map.

Moreover, \hyperref[asm:standing_assumptions]{\autoref*{asm:standing_assumptions}~(1)} implies that Korn's inequality holds on $\Omega$, i.e.,
\begin{equation}
	\label{eq:Korns_inequality}
	\norm{\bu}^2_{H^1(\Omega;\R^d)} \le c_K \, \big( \norm{\bu}^2_{L^2(\Gamma_D;\R^d)}  + \norm{\bvarepsilon(\bu)}^2_S \big)
\end{equation}
for all $\bu \in H^1(\Omega;\R^d)$,
see e.g.\ \citeautoref{HerzogMeyerWachsmuth2009:3:lem:norm}.
Note that \eqref{eq:Korns_inequality} entails in particular that $\norm{\bvarepsilon(\bu)}_S$ is a norm on $H^1_D(\Omega;\R^d)$ equivalent to the standard $H^1(\Omega;\R^d)$ norm.
A further consequence is that $B^\star$ satisfies the inf-sup condition
\begin{equation}
	\label{eq:inf-sup}
	\norm{\bu}_V
	\le
	\sqrt{c_K} \;
	\norm{B^\star \bu}_{S^2}
	\quad\text{for all } \bu \in V.
\end{equation}

\hyperref[asm:standing_assumptions]{\autoref*{asm:standing_assumptions}~(3)} is satisfied, e.g., for isotropic and homogeneous materials, for which
\begin{equation*}
	\C^{-1}\bsigma = \frac{1}{2\,\mu} \bsigma - \frac{\lambda}{2\,\mu\,(2\,\mu+d\,\lambda)} \trace(\bsigma) \, \bI
\end{equation*}
with Lamé constants $\mu$ and $\lambda$, provided that $\mu > 0$ and $d \, \lambda + 2 \, \mu > 0$ hold.
These constants appear only here and there is no risk of confusion with the plastic multiplier $\lambda$.
A common example for the hardening modulus is given by $\H^{-1} \bchi = \bchi / k_1$ with hardening constant $k_1>0$, see \cite[Section~3.4]{HanReddy1999}.

Clearly,
\hyperref[asm:standing_assumptions]{\autoref*{asm:standing_assumptions}~(3)} shows that $\dual{A\bSigma}{\bSigma}_{S^2} \ge \underline{\alpha} \, \norm{\bSigma}^2_{S^2}$ for some $\underline{\alpha} > 0$ and all $\bSigma \in \S^2$.
Hence, the operator $A$ is $S^2$-elliptic.

\subsection{Weak formulation of \eqref{eq:Forward_problem_strong_form} and known results}
\label{subsec:stress_based_formulation}

Testing the strong formulation of the equilibrium of forces
\begin{equation*}
	\div \bsigma(t) = \bnull \quad\text{in } \Omega
	\qquad\text{and}\qquad
	\bsigma(t) \cdot \bn = \bg(t) \quad\text{on }\Gamma_N
\end{equation*}
with $\bv \in V$ and integrating by parts, we obtain 
\begin{equation*}
	\dual{\bSigma(t)}{B^\star \bv}_{S^2}
	= -\int_\Omega \bsigma(t) \dprod \bvarepsilon(\bv) \, \d x
	= -\int_{\Gamma_N} \bg(t) \cdot \bv \, \d s
	= \dual{\ell(t)}{\bv}_{V', V}
\end{equation*}
for all $\bv \in V,$
or equivalently $B\bSigma(t) = \ell(t)$ in $V'$, with $\ell = E \bg$, see \eqref{eq:def_E}.

In order to derive the weak formulation of the first two equations of \eqref{eq:Forward_problem_strong_form},
we fix an arbitrary test function $\bT = (\btau, \bmu) \in \KK$.
Testing the first and second equation of \eqref{eq:Forward_problem_strong_form} with $\btau-\bsigma(t)$ and $\bmu-\bchi(t)$, respectively,
we obtain
\begin{equation*}
	\dual{A \dot\bSigma(t) + B^\star \dot\bu(t)}{\bT - \bSigma(t)}_{S^2}
	+ \int_\Omega \lambda(t) \, \DD\bSigma(t)\dprod \big(\DD\bT - \DD\bSigma(t) \big) \, \d x = 0.
\end{equation*}
Due to the complementarity relation in \eqref{eq:Forward_problem_strong_form} and $\bSigma(t), \bT \in \KK$ we find
\begin{equation*}
	\int_\Omega \lambda(t) \, \DD\bSigma(t)\dprod(\DD\bT - \DD\bSigma(t)) \, \d x
	\le
	\int_\Omega \lambda(t) \, \big( \abs{\DD\bSigma(t)} \, \abs{\DD\bT} - \tilde\sigma_0^2 \big) \, \d x
	\le
	0.
\end{equation*}
Hence, we have
\begin{equation*}
	\dual{A \dot\bSigma(t) + B^\star \dot\bu(t)}{\bT - \bSigma(t)}_{S^2} \ge 0
	\quad \text{for all } \bT \in \KK.
\end{equation*}%

We have derived the weak formulation of \eqref{eq:Forward_problem_strong_form} in the stress-based (so-called dual) form.
It is represented by a time-dependent, rate-independent variational inequality (VI) of mixed type:
find generalized stresses $\bSigma \in H^1(0,T;S^2)$ and displacements $\bu \in H^1(0,T;V)$ which satisfy $\bSigma(t) \in \KK$ and 
\begin{equation}
	\label{eq:VI_lower_level_introduction}
	\tag{$\mathbf{VI}$}
	\begin{aligned}
		\dual{A\dot\bSigma(t)+B^\star\dot\bu(t)}{\bT - \bSigma(t)}_{S^2} &\ge \mrep{0}{\ell(t)} \quad \text{for all } \bT \in \KK, \\
		B\bSigma(t) &= \ell(t) \quad \text{in } V',
	\end{aligned}
\end{equation}
f.a.a.\ $t \in (0,T)$.
Moreover, \eqref{eq:VI_lower_level_introduction} is subject to the initial condition $(\bSigma(0),\bu(0)) = (\bnull,\bnull)$.
In order to guarantee the existence of a solution, we have to require $\ell(0) = \bnull$.
Note that a weak formulation involving the plastic multiplier $\lambda$ is given in \eqref{eq:Lower-Level_Problem_multi}.

The remainder of this section is devoted to known results on the analysis of \eqref{eq:VI_lower_level_introduction}.
We start with the results given in \cite[Section~8]{HanReddy1999}.
We point out that the authors handle a general situation which includes the case of kinematic hardening as a special case.

\subsubsection*{Reformulation as a sweeping process in $\bSigma$}
For $\ell \in V'$, we denote by
\begin{equation}
	\label{eq:KK_ell}
	\KK_\ell = \{ \bSigma \in \KK: B\bSigma = \ell \}
\end{equation}
the subset of $\KK$ on which the constraint $B\bSigma = \ell$ is fulfilled.
Testing the first equation of \eqref{eq:VI_lower_level_introduction} with $\bT \in \KK_{\ell(t)}$
results in:
given $\ell \in H_{\{0\}}^1(0,T;V')$, 
find $\bSigma \in H_{\{0\}}^1(0,T;V')$,
satisfying
$\bSigma(t) \in \KK_{\ell(t)}$
and
\begin{equation}
	\label{eq:VI}
	\dual{A \dot\bSigma(t) }{ \bT - \bSigma(t) }_{S^2} \ge 0 \qquad \text{for all } \bT \in \KK_{\ell(t)} \text{ and almost all } t \in (0,T).
\end{equation}
This is called the stress problem of plasticity.
It is a time-dependent VI, where the associated convex set $\KK_{\ell(t)}$ changes in time.
Such an equation was introduced in \cite{Moreau1977}.
We mention that there are existence and uniqueness results, but no continuity results for the abstract situation considered in \cite{Moreau1977}
seem to be available.

\subsubsection*{Stress problem}
\cite[Section~8]{HanReddy1999} deals with the so-called dual formulation \eqref{eq:VI}.
In \cite[Theorem~8.9]{HanReddy1999} the existence and uniqueness of a solution $\bSigma$ of \eqref{eq:VI} is shown
together with the a-priori bound
\begin{equation}
	\norm{\bSigma}_{H^1(0, T; S^2)} \le C \, \norm{\ell}_{H^1(0, T; V')}.
	\label{eq:a-priori_bsigma}
\end{equation}
Let us denote the solution map by $\GG^\bSigma$, i.e.\ $\bSigma = \GG^\bSigma(\ell)$.
Additionally, \cite[Theorem~8.10]{HanReddy1999} shows the local Hölder continuity of index $1/2$ of $\GG^\bSigma : H_{\{0\}}^1(0, T; V') \to L^\infty(0, T; S^2)$, i.e.,
\begin{equation}
	\label{eq:local_1/2_hoelder_dual}
	\bignorm{\GG^\bSigma(\ell_1) - \GG^\bSigma(\ell_2)}_{L^\infty(0, T; S^2)}^2 \le C \, \big(\norm{\dot\ell_1}_{L^2(0, T; V')} + \norm{\dot\ell_2}_{L^2(0, T; V')}\big) \, \norm{\ell_1 - \ell_2}_{L^2(0, T; V')}.
\end{equation}
We mention that this is not a typical estimate.
As long as the derivatives of $\ell_i$ remain bounded in $L^2(0,T;V')$, one can control the $L^\infty$-norm
of $\bSigma_1-\bSigma_2$ solely through the $L^2$-norm of the difference $\ell_1 - \ell_2$.
We give a generalization of this estimate in \autoref{thm:local_1/2_hoelder_evi} in the context of EVIs.

Finally, \cite[Theorem~8.12]{HanReddy1999} shows that one can introduce a (not necessarily unique) multiplier associated with the equilibrium of forces $B\bSigma = \ell$, which can be interpreted as the displacement $\bu$.
As a consequence, $(\bSigma, \bu)$ satisfies \eqref{eq:VI_lower_level_introduction}.
Thus, given $\ell \in H_{\{0\}}^1(0,T; V')$, there exists $(\bSigma, \bu) \in H_{\{0\}}^1(0,T; S^2 \times V)$ such that $\bSigma(t) \in \KK$ and
\begin{subequations}
	\label{eq:Lower-Level_Problem}
	\begin{align}
		\label{eq:Lower-Level_Problem1}
		\dual{A\dot\bSigma(t)+B^\star\dot\bu(t)}{\bT - \bSigma(t)}_{S^2} &\ge \mrep{0}{\ell(t)} \quad \text{for all } \bT \in \KK, \\
		\label{eq:Lower-Level_Problem2}
		B\bSigma(t) &= \ell(t) \quad \text{in } V'
	\end{align}
\end{subequations}
holds for almost all $t \in (0, T)$.
This shows the equivalence of \eqref{eq:VI} and \eqref{eq:VI_lower_level_introduction}.
Note that \eqref{eq:Lower-Level_Problem} is equivalent to \eqref{eq:VI_lower_level_introduction}.
In the sequel we use either reference as appropriate.

\subsubsection*{Kinematic hardening}
In the case of kinematic hardening, the uniqueness of $\bu$ is obtained easily.
Let us test \eqref{eq:Lower-Level_Problem1} with $\bT = \bSigma(t) + (\btau, -\btau) = (\bsigma(t)+\btau, \bchi(t)-\btau)$.
Due to the shift invariance \eqref{eq:shift-invariance}, we have $\bT \in \KK$ for all $\btau \in S$. This yields
\begin{equation*}
	\dual{ A \dot\bSigma(t) + B^\star \dot\bu(t) }{ (\btau, -\btau) }_{S^2} = 0 \quad \text{for all }\btau \in S.
\end{equation*}
Using the definitions of $A$ and $B$, see \eqref{eq:Definition_of_a} and \eqref{eq:Definition_of_b},
respectively,
we obtain
\begin{equation*}
	\C^{-1} \dot\bsigma - \bvarepsilon(\dot\bu) - \H^{-1} \dot\bchi = \bnull \quad\text{almost everywhere in } (0,T) \times \Omega.
\end{equation*}
Integrating from $0$ to $t$ and using the initial condition $(\bSigma(0), \bu(0)) = \bnull$ yields
\begin{equation}
	\label{eq:relation_sigma_chi_u}
	\C^{-1} \bsigma - \bvarepsilon(\bu) - \H^{-1} \bchi = \bnull \quad\text{almost everywhere in } (0,T) \times \Omega.
\end{equation}
Together with the inf-sup condition of $B^\star = (-\bvarepsilon, \bnull)$, see \eqref{eq:inf-sup}, this proves the uniqueness of $\bu$.
Using \eqref{eq:a-priori_bsigma} yields the a-priori estimate
\begin{equation}
	\norm{\bSigma}_{H^1(0, T; S^2)} + \norm{\bu}_{H^1(0, T; V)} \le C \, \norm{\ell}_{H^1(0, T; V')}.
	\label{eq:a-priori_bsigma_bu}
\end{equation}
We denote the solution mapping $\ell \mapsto \bu$ of \eqref{eq:Lower-Level_Problem} by $\GG^\bu$.
Moreover, the solution operator of \eqref{eq:VI_lower_level_introduction} mapping $\ell \to (\bSigma, \bu)$ is denoted by $\GG = (\GG^\bSigma, \GG^\bu)$.

\subsubsection*{Primal problem}
In \cite[Section~7]{HanReddy1999} the primal formulation of \eqref{eq:VI} is considered.
Both formulations are equivalent, see \cite[Theorem~8.3]{HanReddy1999}.
For the primal formulation and under the (additional) assumption of kinematic (or combined kinematic-isotropic) hardening, the (global) Lipschitz continuity of the solution operator from $W^{1,1}(0,T; Z')$ to $L^{\infty}(0, T; Z)$ was proven, see \cite[pp.~170--171]{HanReddy1999}.
Here, $Z$ is the appropriate function space for the analysis of the primal formulation.
Due to the equivalence of the primal and the dual problem, this Lipschitz estimate carries over to the dual problem.
We obtain
\begin{equation}
	\label{eq:lipschitz_continuity_W11_Linfty}
	\norm{\GG(\ell_1) - \GG(\ell_2)}_{L^\infty(0,T; S^2\times V)} \le C \, \norm{\ell_1 - \ell_2}_{W^{1,1}(0,T;V')}.
\end{equation}
This Lipschitz estimate was actually already contained in rather unknown works of Gröger, see \cite{Groeger1978:1,Groeger1978:2,Groeger1978}.
In \cite[Section~4]{Groeger1978} the system \eqref{eq:VI_lower_level_introduction} is reformulated as an evolution equation associated with a maximal monotone operator.
Then the Lipschitz estimate follows by the classical result \cite[Lemma~3.1]{Brezis1973}.

\begin{remark}
	\label{rem:continuity_not_sufficient}
	In order to derive optimality conditions, the continuity results mentioned above are not sufficient.
	In addition, we need the continuity of $\GG : H_{\{0\}}^1(0,T;V') \to H_{\{0\}}^1(0,T;S^2 \times V)$ on two occasions.
	First, this continuity is needed to prove the approximability of local solutions by time-discrete minimizers in \autoref{thm:continuous_approximation_with_global_solutions}.
	Second, the strong convergence of $(\dot\bSigma, \dot\bu)$ in $L^2(0,T;S^2 \times V)$
	is needed to pass to the limit in the optimality system, see \citeautoref{Wachsmuth2011:4:sec:weak_stationarity_quasistatic}.
\end{remark}
The required continuity of $\GG : H_{\{0\}}^1(0,T;V') \to H_{\{0\}}^1(0,T;S^2\times V)$ is shown in \autoref{cor:continuity_GG_in_H1} by building on some 
results of \cite{Krejci1996,Krejci1998} about evolution variational inequalities (EVIs).
To our knowledge, this is a new result for the analysis of \eqref{eq:VI_lower_level_introduction}.

\subsubsection*{Existence, uniqueness and regularity of the plastic multiplier}

We have already seen that the generalized plastic strain $\bP = -A\bSigma - B^\star \bu$ is unique.
Thus by \cite{HerzogMeyerWachsmuth2010:1} we obtain the existence and uniqueness of the plastic multiplier $\lambda \in L^2(0,T; L^2(\Omega))$ which can be understood as a multiplier associated with the constraint $\bSigma \in \KK$ or rather $\phi(\bSigma) \le 0$.
We obtain the system
\begin{subequations}
	\label{eq:Lower-Level_Problem_multi}
	\begin{align}
		A \dot\bSigma + B^\star \dot\bu + \lambda \, \DD^\star \DD \bSigma &= \mrep{\bnull}{\ell} \quad \text{in } L^2(0,T; S^2), \label{eq:Lower-Level_Problem_multi1} \\
		B\bSigma &= \ell \quad \text{in } L^2(0,T;V'), \label{eq:Lower-Level_Problem_multi2}\\
		0 \le \lambda \quad \perp \quad \phi(\bSigma) & \le \mrep{0}{\ell} \quad \text{a.e.\ in } (0,T) \times \Omega.
	\end{align}
\end{subequations}
Endowed with the initial condition $(\bSigma(0),\bu(0)) = (\bnull,\bnull)$,
this system is equivalent to \eqref{eq:VI} and \eqref{eq:Lower-Level_Problem}.

\section{Analysis of the continuous optimal control problem}
\label{sec:continuous}
In this section we study the continuous (i.e.\ non-discretized) optimal control problem.
In the first subsection, we give a brief introduction to evolution variational inequalities (EVIs) and state the continuity of the solution map of an EVI from $H^1(0,T;X)$ into itself, see \autoref{thm:solution_of_evi}.
This enables us to show the continuity of the solution map of \eqref{eq:VI_lower_level_introduction}, see \autoref{sec:evolution_vi}.
\autoref{subsec:weak_continuity} is devoted to show the weak continuity of the control-to-state map of \eqref{eq:upper_level_in_intro}.
Due to this weak continuity, we are able to conclude the existence of an optimal control in \autoref{subsec:continuous_ulp}, see \autoref{thm:existence_continuous}.

\subsection{Introduction to evolution variational inequalities}
In this section we give a short introduction to \emph{evolution variational inequalities}.
A comprehensive presentation of this topic can be found in \cite{Krejci1998}.
We also mention the contributions \cite{KrejciLovicar1990,Krejci1996,BrokateKrejci1998}.
For convenience of the reader who wishes to consult \cite{Krejci1998} in parallel, we use the notation of \cite{Krejci1998} in this subsection and turn back to our notation in the next subsection.

Let $X$ be some Hilbert space and $Z \subset X$ be some convex, closed set.
Given a function $u : [0,T] \to X$ and an initial value $x_0 \in Z$, find $x : [0,T] \to Z \subset X$ such that
$x(0) = x_0$ and
\begin{equation}
	\label{eq:abstract_evi}
	\tag{$\mathbf{EVI}$}
	\scalarprod{\dot u(t) - \dot x(t)}{x(t) - \tilde x}_X \ge 0
	\quad
	\text{for all } \tilde x \in Z.
\end{equation}
There are many applications that lead to an EVI, see \cite[page~2]{Krejci1998} for references and further comments on this topic.
We only mention the strain or stress driven problem of plasticity with linear kinematic hardening.
In the strain (stress) driven problem, the strain (stress) is viewed as a known quantity, whereas the stress (strain) has to be determined.
We show that the dual formulation \eqref{eq:Lower-Level_Problem} is an appropriate starting point for deriving the strain driven problem.
Krejčí derived the stress and the strain driven problems in \cite[(1.29) and (1.31)]{Krejci1998}, respectively.

In the strain driven problem, the strain $\bvarepsilon(\bu)$ (or, equivalently, the strain rate $(\bvarepsilon(\dot\bu),\bnull) = -B^\star \dot\bu$) is viewed as a given quantity.
The variational inequality \eqref{eq:Lower-Level_Problem1} 
reduces to find a function $\bSigma : [0,T] \to \KK \subset S^2$, such that
\begin{equation}
	\label{eq:strain_driven_plasticity}
	\vdual{ \dot\bSigma(t) -
		\begin{pmatrix}
			\C \, \bvarepsilon(\dot\bu(t)) \\ 0
		\end{pmatrix}
	}{\bT - \bSigma(t)}_A \ge 0
	\quad
	\text{for all } \bT \in \KK,
\end{equation}
where $\dual{\cdot}{\cdot}_A$ is the scalar product on $S^2$ induced by $A$, i.e.,
\begin{equation}
	\label{eq:norm_A}
	\dual{\bSigma}{\bT}_A := \dual{A\bSigma}{\bT}_{S^2}
	\quad\text{for }\bSigma,\bT \in S^2.
\end{equation}
Using the coercivity of $A$, which is ensured by \hyperref[asm:standing_assumptions]{\autoref*{asm:standing_assumptions}~(3)},
we find that $\dual{\cdot}{\cdot}_A$ is scalar product on $S^2$ which is equivalent to the standard scalar product.
The equation \eqref{eq:strain_driven_plasticity} is of the form \eqref{eq:abstract_evi} in the Hilbert space $X = S^2$ equipped with the scalar product $\dual{\cdot}{\cdot}_A$ and the feasible set $Z = \KK$.

\begin{remark}
	\label{rem:differences_evi_strain_driven}
	We remark that there is a fundamental difference between the dual formulation \eqref{eq:Lower-Level_Problem} and the strain driven problem \eqref{eq:strain_driven_plasticity}. Whereas in \eqref{eq:Lower-Level_Problem} \emph{both} $\bSigma$ and $\bu$ are unknown quantities, $\bvarepsilon(\bu)$ has to be \emph{a-priori known} in \eqref{eq:strain_driven_plasticity}. Moreover, the solution $\bSigma$ of \eqref{eq:strain_driven_plasticity} has to fulfill \eqref{eq:Lower-Level_Problem2}.

	Let us denote the solution map of \eqref{eq:strain_driven_plasticity} which maps $\bu \to \bSigma$ by $\SS$.
	Therefore, \eqref{eq:Lower-Level_Problem2} yields $\ell = B\bSigma = B\SS(\bu)$. Thus, \eqref{eq:Lower-Level_Problem} is equivalent to
	\begin{equation*}
		\text{Find } \bu \text{ such that } B\SS(\bu) = \ell.
	\end{equation*}
	This technique is used in \cite[Chapter~III]{Krejci1996} for hyperbolic equations arising in plasticity, i.e., the author considers plasticity problems where the acceleration term $\rho \, \ddot\bu$ is included in the balance of forces \eqref{eq:Lower-Level_Problem2}.
\end{remark}

The following theorem summarizes results on \eqref{eq:abstract_evi} given in \cite{Krejci1996,Krejci1998}.
\begin{theorem}
	\label{thm:solution_of_evi}
	Let $u \in W^{1,1}(0, T; X)$ and $x_0 \in Z$ be given. Then there exists a unique solution $S(x_0, u) := x \in W^{1,1}(0, T; X)$ of \eqref{eq:abstract_evi}.
	Moreover, $S : Z \times W^{1,1}(0, T; X) \to L^\infty(0, T; X)$ is globally Lipschitz continuous
	and $S : Z \times W^{1,p}(0, T; X) \to W^{1,p}(0, T; X)$ is continuous for all $p < \infty$.
\end{theorem}
The operator $S$ is called the \emph{stop operator},
whereas $P(x_0,u) = u - S(x_0,u)$ is called the \emph{play operator}.

Let us show a result which generalizes the local Hölder estimate \eqref{eq:local_1/2_hoelder_dual}.
In the context of EVIs this is a new result.
It shows that the play operator $P$ possesses an additional smoothing property which is not shared by the stop operator $S$.
Another such property is that $P$ maps $C(0,T;X)$ (continuous functions mapping $[0,T] \to X$) to $CBV(0,T;X)$ (continuous functions of bounded variation mapping $[0,T] \to X$), see \cite[Theorem~3.11]{Krejci1998}.
Again, this does not hold for the stop operator $S$.
\begin{theorem}
	\label{thm:local_1/2_hoelder_evi}
	Let $p \in [1, \infty]$ be given and denote by $q$ its dual exponent.
	Let $u_1, u_2 \in W^{1,q}(0,T; X)$ and initial values $x_1^0, x_2^0 \in Z$ be given.
	Let $\xi_i = P(x_i^0, u_i)$, $i = 1,2$.
	Then the estimate
	\begin{align}
		\norm{\xi_1-\xi_2}_{L^\infty(0,T; X)}^2
		&\le
		2 \, \big( \norm{\dot u_1}_{L^q(0,T; X)} + \norm{\dot u_2}_{L^q(0,T; X)} \big) \norm{ u_1 - u_2 }_{L^p(0,T; X)}
		\notag
		\\
		&
		\notag
		\qquad
		+
		\norm{\xi_1(0) - \xi_2(0)}_X^2
	\end{align}
	holds, where $\xi_i(0) = u_i(0) - x_i^0$, $i = 1,2$.
\end{theorem}
\begin{proof}
	First we recall \cite[(3.16ii)]{Krejci1998}, i.e., $\scalarprod{\dot\xi_1(t)}{\dot x_1(t)}_X = 0$ a.e.\ in $(0,T)$,
	where $x_1 = S(x_1^0,u_1)$.
	Using $x_1 = u_1 - \xi_1$, this implies $\norm{\dot\xi_1(t)}_X \le \norm{\dot u_1(t)}_X$ a.e.\ in $(0,T)$.
	Analogously, we obtain $\norm{\dot\xi_2(t)}_X \le \norm{\dot u_2(t)}_X$ a.e.\ in $(0,T)$.

	Using $\tilde x = u_2(t) - \xi_2(t) = x_2(t) \in Z$ as a test function in \eqref{eq:abstract_evi} shows
	\begin{equation*}
		\scalarprod{\dot\xi_1(t)}{u_1(t) - u_2(t) + \xi_2(t) - \xi_1(t)}_X \ge 0 \quad\text{a.e.\ in }(0,T).
	\end{equation*}
	Similarly, we obtain
	\begin{equation*}
		\scalarprod{\dot\xi_2(t)}{u_2(t) - u_1(t) + \xi_1(t) - \xi_2(t)}_X \ge 0 \quad\text{a.e.\ in }(0,T).
	\end{equation*}
	Adding these inequalities, we have
	\begin{equation*}
		\scalarprod{\dot\xi_1(t) - \dot\xi_2(t)}{\xi_1(t) - \xi_2(t)}_X \le \scalarprod{\dot\xi_1(t) - \dot\xi_2(t)}{u_1(t) - u_2(t)}_X \quad\text{a.e.\ in }(0,T).
	\end{equation*}
	This shows
	\begin{align*}
		\frac12 \, \frac{\d}{\d t} \, \norm{\xi_1 - \xi_2}_X^2
		& = \scalarprod{\dot\xi_1(t) - \dot\xi_2(t)}{\xi_1(t) - \xi_2(t)}_X \\
		& \le \scalarprod{\dot\xi_1(t) - \dot\xi_2(t)}{u_1(t) - u_2(t)}_X \\
		& \le (\norm{\dot\xi_1(t)}_X + \norm{\dot\xi_2(t)}_X) \, \norm{u_1(t) - u_2(t)}_X \\
		& \le (\norm{\dot u_1(t)}_X + \norm{\dot u_2(t)}_X) \, \norm{u_1(t) - u_2(t)}_X
		\quad\text{a.e.\ in }(0,T).
	\end{align*}
	Integrating from $0$ to $t$ yields
	\begin{align*}
		& \frac12 \, \norm{\xi_1(t) - \xi_2 (t)}_X^2 - \frac12 \, \norm{\xi_1(0) - \xi_2(0)}_X^2 \\
		& \qquad \le \int_0^t (\norm{\dot u_1(s)}_X + \norm{\dot u_2(s)}_X) \, \norm{u_1(s) - u_2(s)}_X \, \d s \\
		& \qquad \le (\norm{\dot u_1}_{L^q(0,T; X)} + \norm{\dot u_2}_{L^q(0,T; X)}) \, \norm{u_1 - u_2}_{L^p(0,T; X)}.
	\end{align*}
	Taking the supremum $t \in (0,T)$ yields the claim.
\end{proof}

\subsection{Quasistatic plasticity as an EVI}
\label{sec:evolution_vi}
In this section we give an equivalent reformulation of \eqref{eq:Lower-Level_Problem} which fits into the framework of \cite{Krejci1998}.
Once this reformulation is established, the continuity of $\GG : H_{\{0\}}^1(0, T; V') \to H_{\{0\}}^1(0, T; S^2 \times V)$ is a consequence of \autoref{thm:solution_of_evi}.
As mentioned in \autoref{rem:continuity_not_sufficient}, the continuity of $\GG$ is used in several places.

In order to reformulate \eqref{eq:Lower-Level_Problem}, we need some preparatory work.
Due to the inf-sup condition \eqref{eq:inf-sup} we obtain that for all $\ell \in V'$ there is a $\bsigma_\ell \in S$, such that
\begin{equation}
	\label{eq:bsigma(ell)}
	(\bsigma_\ell,\bnull) \in (\ker B)^\perp \quad\text{and}\quad B(\bsigma_\ell,\bnull) = \ell,
\end{equation}
see \cite{Brezzi1974}.
Moreover, this mapping $\ell \mapsto \bsigma_\ell$ is linear and continuous.
For $\ell \in H_{\{0\}}^1(0,T;V')$, we define $\bsigma_\ell(t) := \bsigma_{\ell(t)}$. Hence, $\bsigma_{(\cdot)}$ maps $H_{\{0\}}^1(0,T;V') \to H_{\{0\}}^1(0,T;S)$ continuously.
For convenience, we also introduce the notation 
\begin{equation}
	\label{eq:bSigma(ell)}
	\bSigma_\ell = (\bsigma_\ell, -\bsigma_\ell).
\end{equation}

Let us mention that this is a useful tool to construct test functions $\bT$ for \eqref{eq:VI} and \eqref{eq:Lower-Level_Problem}.
Indeed, for arbitrary $\bSigma \in \KK$ and $\ell \in V'$, let us define the test function $\bT = \bSigma + \bSigma_{\ell - B\bSigma}$.
The shift invariance \eqref{eq:shift-invariance} implies $\bT \in \KK$
and $B\bT = \ell$ is ensured by the definition of $\bSigma_\ell$.
Hence we obtain $\bT \in \KK_\ell$, i.e., it is an admissible test function for \eqref{eq:VI} and \eqref{eq:Lower-Level_Problem}.

Let $\ell \in H_{\{0\}}^1(0,T;V')$ be arbitrary.
Using \eqref{eq:bSigma(ell)} we are able to reformulate \eqref{eq:VI} as an EVI.
As stated in \eqref{eq:shift-invariance}, $\bSigma \in \KK$ holds if and only if $\bSigma + (\btau, -\btau) \in \KK$ holds for all $\btau \in S$.
Since $\bSigma_\ell = (\bsigma_\ell, -\bsigma_\ell)$ by definition, $\bSigma \in \KK$ if and only if $\bSigma - \bSigma_\ell \in \KK$.
This gives rise to the decomposition $\bSigma = \bSigma_0 + \bSigma_\ell$.
There are two immediate consequences: $\bSigma \in \KK$ is equivalent to $\bSigma_0 \in \KK$ and $B\bSigma = \ell$ is equivalent to $\bSigma_0 \in \ker B$.
Therefore, $\bSigma_0 \in \KK \cap \ker B$.

We define $\KK_B := \KK \cap \ker B$. Let $\bT \in \KK_B$ be arbitrary. Testing \eqref{eq:Lower-Level_Problem1} or \eqref{eq:VI} with $\bT + \bSigma_{\ell}(t) \in \KK$ yields an equivalent reformulation: find $\bSigma_0 \in \KK_B$ such that
\begin{equation}
	\dual{ A( \dot\bSigma_0 + \dot\bSigma_\ell) }{ \bT - \bSigma_0} \ge 0 \quad \text{for all } \bT \in \KK_B.
	\label{eq:eq:Lower-Level_Problem_evolution_vi}
\end{equation}
Hence, \eqref{eq:eq:Lower-Level_Problem_evolution_vi} is an EVI in the Hilbert space $S^2$ equipped with $\dual{\cdot}{\cdot}_A$.
\autoref{thm:solution_of_evi} yields the continuity of the mapping $\bSigma_\ell \mapsto \bSigma_0$ from $H_{\{0\}}^1(0, T; S^2)$ to $H_{\{0\}}^1(0, T; S^2)$.
Since $\ell \mapsto \bSigma_\ell$ is continuous from $H_{\{0\}}^1(0,T;V') \to H_{\{0\}}^1(0,T;S^2)$,
we obtain the continuity of the mapping $\ell \mapsto \bSigma$ from $H_{\{0\}}^1(0,T; V')$ to $H_{\{0\}}^1(0,T; S^2)$.

The uniqueness and continuous dependence of $\bu$ on $\ell$ follows from \eqref{eq:relation_sigma_chi_u} and the inf-sup condition of $B^\star$, see \eqref{eq:inf-sup}.
Moreover, \autoref{thm:solution_of_evi} also shows the Lipschitz continuity of $\GG: \ell \mapsto (\bSigma, \bu)$, $W^{1,1}(0,T; V') \to L^\infty(0,T; S^2 \times V)$, see \eqref{eq:lipschitz_continuity_W11_Linfty}.
Summarizing, we have found
\begin{corollary}
	\label{cor:continuity_GG_in_H1}
	The solution operator $\GG$ of \eqref{eq:VI_lower_level_introduction} is continuous from $H_{\{0\}}^1(0,T;V')$ to $H_{\{0\}}^1(0,T;S^2\times V)$
	and globally Lipschitz continuous to $L^\infty(0,T;S^2 \times V)$.
\end{corollary}

\begin{remark}
	The technique presented in this subsection is not restricted to kinematic hardening only, the arguments remain valid for all hardening models involving some kinematic part, see also \cite{Groeger1978}.
	To be precise, the technique is applicable whenever the hardening variable can be split into a kinematic part $\bchi$ and some other part $\boldeta$,
	such that
	the yield function $\phi$ can be written as $\phi(\bSigma) = \tilde\phi( \bsigma+\bchi, \boldeta )$,
	where $\bSigma = (\bsigma, \bchi, \boldeta)$.

	It is interesting to note that this is exactly the case for which \cite{HanReddy1999} were able to prove Lipschitz continuity of the forward operator (in primal formulation) from $W^{1,1}(0,T; Z')$ to $L^\infty(0,T; Z)$, see the discussion in \autoref{subsec:stress_based_formulation}.
\end{remark}
\subsection{Weak continuity of the forward operator}
\label{subsec:weak_continuity}
In this section we show the weak continuity of the control-to-state map $\GG \circ E$
of the optimal control problem \eqref{eq:upper_level_in_intro}.
Here, $\GG$ is the solution map of \eqref{eq:Lower-Level_Problem}
and $E$ is the control operator, see \eqref{eq:def_E}. 
Since the solution operator $\GG \circ E$ is nonlinear, the proof of its weak continuity is non-trivial.
The weak continuity of $\GG \circ E$ from $H_{\{0\}}^1(0, T; U)$ to $H_{\{0\}}^1(0, T; S^2 \times V)$ is essential for proving the existence of an optimal control in \autoref{thm:existence_continuous}.

A main ingredient to prove the weak continuity is 
the compactness of the control operator $E$ from $U = L^2(\Gamma_N; \R^d)$ to $V' = (H^1_D(\Omega; \R^d))'$, see the proof of \autoref{lem:limsup}.
Due to this compactness, Aubin's lemma, see, e.g., \cite[Equation~(6.5)]{Simon1986}, implies that $H^1(0,T; U)$ embeds compactly into $L^2(0,T; V')$.

Let us mention that the weak continuity of $\GG \circ E$ is a non-trivial result:
although $E : U \to V'$ is compact, the operator $E : H^1(0, T; U) \to H^1(0, T; V')$ (applied in a pointwise sense) is not compact.
Hence the weak continuity of $\GG \circ E$ does not simply follow from the compactness of $E$ nor the continuity of $\GG$.

First, we need a preliminary result. It can be interpreted as an upper semicontinuity result for the left-hand side of \eqref{eq:Lower-Level_Problem1}.
\begin{lemma}
	Let $\{(\bSigma_k, \bu_k)\} \subset H_{\{0\}}^1(0, T; S^2 \times V)$
	and $(\bSigma, \bu) \in H^1(0, T; S^2 \times V)$ be given.
	Moreover, we assume that there are $\bg_k, \bg \in H^1(0,T; U)$ such that $B\bSigma_k = E \bg_k$ and $B\bSigma = E \bg$ hold.
	If
	\begin{align*}
		(\bSigma_k, \bu_k) &\weakly (\bSigma, \bu) \text{ in } H^1(0, T; S^2 \times V)  \quad \text{and}\\
		\bg_k &\weakly \mrep{\bg}{(\bSigma, \bu)} \text{ in } H^1(0, T; U),
	\end{align*}
	then for all $\bT \in L^2(0, T; S^2)$ we have
	\begin{equation*}
		\limsup_{k \to \infty} \int_0^T \dual{A \dot\bSigma_k + B^\star \dot\bu_k }{ \bT - \bSigma_k } \, \d t
		\le
		\int_0^T \dual{A \dot{\bSigma} + B^\star \dot{\bu} }{ \bT - \bSigma } \, \d t.
	\end{equation*}
	\label{lem:limsup}
\end{lemma}
\begin{proof}
	Due to the weak convergence of $(\bSigma_k, \bu_k)$ in $H^1(0,T; S^2 \times V)$ we have
	\begin{equation*}
		\int_0^T \dual{A \dot\bSigma_k + B^\star \dot\bu_k }{ \bT } \, \d t
		\to
		\int_0^T \dual{A \dot{\bSigma} + B^\star \dot{\bu} }{ \bT } \, \d t.
	\end{equation*}

	Since $U$ embeds compactly into $V'$, Aubin's lemma yields that $H^1(0, T; U)$ embeds compactly into $L^2(0, T; V')$ and hence
	$\bg_k \weakly \bg$ in $H^1(0, T; U)$ implies
	$E \bg_k \to E \bg$ in $L^2(0, T; V')$.
	This shows
	\begin{align*}
		\int_0^T \dual{B^\star \dot\bu_k }{ \bSigma_k } \, \d t
		&=
		\int_0^T \dual{\dot\bu_k }{ B \bSigma_k } \, \d t
		=
		\int_0^T \dual{\dot\bu_k }{ E \bg_k } \, \d t \\
		&\to
		\int_0^T \dual{\dot{\bu} }{ E \bg } \, \d t
		=
		\int_0^T \dual{\dot{\bu} }{ B\bSigma } \, \d t
		=
		\int_0^T \dual{B^\star \dot{\bu} }{ \bSigma } \, \d t.
	\end{align*}

	To address the remaining term, we use integration by parts and obtain
	\begin{equation*}
		\int_0^T \dual{ A\dot\bSigma_k }{ \bSigma_k } \, \d t = \frac12 \big( \dual{A\bSigma_k(T)}{\bSigma_k(T)} - \dual{A\bSigma_k(0)}{\bSigma_k(0)} \big).
	\end{equation*}
	The functionals $\bSigma \mapsto \dual{A\bSigma(t)}{\bSigma(t)}$ are continuous (w.r.t.\ the $H^1(0,T;S^2)$-norm) and convex, hence weakly lower semicontinuous. Due to $\bSigma_k(0) = \bnull$, this already implies $\dual{A\bSigma(0)}{\bSigma(0)} = 0$.
	Now the weak lower semicontinuity of $\bSigma \mapsto \dual{A\bSigma(T)}{\bSigma(T)}$ yields
	\begin{equation*}
		\liminf_{k\to\infty} \dual{A\bSigma_k(T)}{\bSigma_k(T)} \ge \dual{A\bSigma(T)}{\bSigma(T)}.
	\end{equation*}
	Therefore,
	\begin{equation*}
		\liminf_{k\to\infty} \int_0^T \dual{ A\dot\bSigma_k }{ \bSigma_k } \, \d t
		\ge
		\frac12 \dual{A\bSigma(T)}{\bSigma(T)}
		=
		\int_0^T \dual{ A\dot{\bSigma} }{ \bSigma } \, \d t
	\end{equation*}
	holds.
	By combining the results above, we obtain the assertion
	\begin{equation*}
		\limsup_{k\to\infty} \int_0^T \dual{A \dot\bSigma_k + B^\star \dot\bu_k }{ \bT - \bSigma_k } \, \d t
		\le
		\int_0^T \dual{A \dot{\bSigma} + B^\star \dot{\bu} }{ \bT - \bSigma } \, \d t.
	\end{equation*}
\end{proof}

The previous lemma enables us to prove

\begin{theorem}
	The operator $\GG \circ E$ is weakly continuous from $H_{\{0\}}^1(0, T; U)$ to $H_{\{0\}}^1(0, T; S^2 \times V)$.
	\label{thm:weak_continuity_solution_operator}
\end{theorem}
\begin{proof}
	Let $\{ \bg_k \} \subset H_{\{0\}}^1(0, T; U)$ be a sequence converging weakly towards $\bg \in H_{\{0\}}^1(0, T; U)$.
	The solution of \eqref{eq:Lower-Level_Problem} with right-hand side $E\bg$, $E\bg_k$ is denoted by $(\bSigma, \bu)$, $(\bSigma_k, \bu_k)$, respectively,
	i.e.,
	$(\bSigma_k, \bu_k) = \GG( E \bg_k)$ and $(\bSigma, \bu) = \GG( E \bg )$.
	Due to the boundedness of $E$ and $\GG$, see \eqref{eq:a-priori_bsigma_bu}, there exists a subsequence (for simplicity denoted by the same symbol) $\{ ( \bSigma_k, \bu_k ) \}$ which converges weakly towards some $(\tilde\bSigma, \tilde\bu)$ in $H^1(0,T;S^2 \times V)$.
	We shall prove $(\tilde\bSigma, \tilde\bu) = (\bSigma, \bu)$, therefore the whole sequence converges weakly and this simplification of notation is justified.
	To this end, we show that $(\tilde\bSigma, \tilde\bu)$ is a solution to \eqref{eq:Lower-Level_Problem} with right-hand side $E\bg$.

	First we address the admissibility of $\tilde\bSigma$ in \eqref{eq:Lower-Level_Problem}:
	the set $\{ \bT \in H^1(0, T; S^2): \bT(t) \in \KK \text{ f.a.a.\ } t \in (0,T)\}$ is convex and closed, hence weakly closed.
	Since all $\bSigma_k$ belong to this set, $\tilde\bSigma$ belongs to it as well.
	Therefore $\tilde\bSigma(t) \in \KK$ holds for almost all $t \in (0,T)$.

	Since $B$ is linear and bounded, it is weakly continuous and hence $B\bSigma_k \weakly B\tilde\bSigma$ in $H^1(0, T; V')$.
	Due to $B\bSigma_k = E\bg_k \weakly E\bg$ in $H^1(0,T; V')$, we have $B\tilde\bSigma = E\bg$ in $H^1(0, T; V')$, this shows \eqref{eq:Lower-Level_Problem2}.

	Let $\bT \in L^2(0, T; S^2)$ with $\bT(t) \in \KK$ f.a.a.\ $t \in (0,T)$ be given.
	This implies
	\begin{equation*}
		\dual{A \dot\bSigma_k(t) + B^\star \dot\bu_k(t) }{ \bT(t) - \bSigma_k(t) } \ge 0
		\quad\text{f.a.a.\ } t \in [0,T],
	\end{equation*}
	see \eqref{eq:Lower-Level_Problem1}.
	Integrating w.r.t.\ $t$ yields
	\begin{equation*}
		\int_0^T \dual{A \dot\bSigma_k + B^\star \dot\bu_k }{ \bT - \bSigma_k } \, \d t \ge 0.
	\end{equation*}
	By applying \autoref{lem:limsup}, we obtain
	\begin{equation*}
		\int_0^T \dual{A \dot{\tilde\bSigma} + B^\star \dot{\tilde\bu} }{ \bT - \tilde\bSigma } \, \d t
		\ge
		\limsup_{k\to\infty} \int_0^T \dual{A \dot\bSigma_k + B^\star \dot\bu_k }{ \bT - \bSigma_k } \, \d t
		\ge
		0.
	\end{equation*}
	Since $\bT \in L^2(0, T; S^2)$ with $\bT(t) \in \KK$ f.a.a.\ $t \in (0,T)$ was arbitrary,
	\begin{equation}
		\label{eq:in_the_proof_of_some_thm}
		\dual{A \dot{\tilde\bSigma}(t) + B^\star \dot{\tilde\bu}(t) }{ \bT - \tilde\bSigma(t) }
		\ge
		0
		\quad\text{for all } \bT \in \KK
	\end{equation}
	holds for every common Lebesgue point $t$ of the functions
	\begin{align*}
		t &\mapsto \dual{A \dot{\tilde\bSigma}(t) + B^\star \dot{\tilde\bu}(t) }{ \tilde\bSigma(t) } \in L^1(0,T; \R) \\
		\text{and} \quad
		t &\mapsto \phantom{\langle{}} A \dot{\tilde\bSigma}(t) + B^\star \dot{\tilde\bu}(t) \in L^2(0,T; S^2).
	\end{align*}
	Lebesgue's differentiation theorem, see \cite[Theorem~V.5.2]{Yosida1965}, yields that almost all $t \in (0,T)$
	are Lebesgue points of these functions.
	This implies that \eqref{eq:in_the_proof_of_some_thm} holds for almost all $t \in (0,T)$.
	Hence $(\tilde\bSigma, \tilde\bu)$ solves \eqref{eq:Lower-Level_Problem}.
	Since the solution is unique we obtain $(\tilde\bSigma, \tilde\bu) = (\bSigma, \bu)$.
	This implies that the weak limit is indeed independent of the chosen subsequence.
	Hence, the whole sequence converges weakly, i.e.\ $(\bSigma_k, \bu_k) \weakly (\bSigma, \bu)$ in $H^1(0,T;S^2 \times V)$.
	That is, $\GG(E\bg_k) = (\bSigma_k, \bu_k) \weakly (\bSigma, \bu) = \GG(E\bg)$ in $H^1(0,T;S^2 \times V)$.
	Since $\{\bg_k\}$ and $\bg$ were arbitrary, $\GG \circ E$ is weakly continuous.
\end{proof}

\subsection{Existence of optimal controls}
\label{subsec:continuous_ulp}
In this section we prove the existence of an optimal control.
The key tool in the proof of \autoref{thm:existence_continuous} is the weak continuity of the forward operator proven in \autoref{thm:weak_continuity_solution_operator}.
For convenience, we repeat the definition of the optimal control problem under consideration
\begin{equation}
	\label{eq:continuous_ulp}
	\tag{$\mathbf{P}$}
	\left.
		\begin{aligned}
			\text{Minimize}\quad & F(\bu,\bg) = \psi( \bu ) + \frac{\nu}{2} \norm{\bg}_{H^1(0,T;U)}^2 \\
			\text{such that}\quad & 
			(\bSigma, \bu) = \GG(E\bg)
			\\
			\text{and}\quad & \bg \in \Uad.
		\end{aligned}
	\quad\right\}
\end{equation}
The admissible set $\Uad$ is a convex closed subset of $H_{\{0\}}^1(0,T;U)$. 
Here, $\norm{\cdot}_{H^1(0,T;U)}$ can be any norm such that $H^1(0,T;U)$ is a Hilbert space, see \eqref{eq:norms_in_w1p}.
Let us fix the assumptions on the objective $\psi$ and on the set of admissible controls $\Uad$.
\begin{assumption}
	\label{asm:psi_lsc}
	\hfill
	\begin{enumerate}
		\item 
			The function $\psi : H^1(0,T; V) \to \R$ is weakly lower semicontinuous, continuous and bounded from below.
		\item 
			The cost parameter $\nu$ is a positive, real number.
		\item 
			The admissible set $\Uad$ is nonempty, convex and closed in $H_{\{0\}}^1(0,T;U)$.
	\end{enumerate}
\end{assumption}
Let us give some examples which satisfy these conditions.
Since the domain of definition of $\psi$ is $H^1(0,T;V)$ we can track the displacements,
the strains or even their time derivatives (or point evaluations).
As examples for $\psi$ we mention
\begin{align*}
	\psi^1(\bu) &= \frac12\,\norm{\bu - \bu_d}_{L^2(0,T; L^2(\Omega; \R^d))}^2,
	&
	\psi^2(\bu) &= \frac12 \, \norm{\bu(T) - \bu_{T,d}}_{L^2(\Omega; \R^d)}^2,
	\\
	\psi^3(\bu) &= \frac12 \, \norm{\bvarepsilon(\bu(T)) - \bvarepsilon_{T,d}}_{L^2(\Omega; \S)}^2.
\end{align*}
Here, $\bu_d \in L^2(0,T;L^2(\Omega;\R^d))$, $\bu_{T,d} \in L^2(\Omega; \R^d)$ and $\bvarepsilon_{T,d} \in L^2(\Omega; \S)$ are desired displacements and strains, respectively.

Let us give some examples for the admissible set $\Uad$.
The sets
\begin{align*}
	\Uad^1 &= \{ \bg \in H^1(0,T;U) : \bg(0) = \bnull, \norm{\bg}_{L^2(\Gamma_N; \R^d)} \le \rho \text{ a.e.\ in } (0,T) \} \\
	\Uad^2 &= \{ \bg \in H^1(0,T;U) : \bg(0) = \bg(T) = \bnull \}
\end{align*}
satisfy \autoref{asm:psi_lsc} for every $\rho \ge 0$.

Of special interest is the combination of $\psi^2$ and $\Uad^2$.
Since we consider a quasistatic process,
the condition $\bg(T) = \bnull$ implies that
at time $t = T$ the solid body is unloaded.
Due to the plastic behaviour, the remaining (and lasting) deformation $\bu(T)$
is typically non-zero.
Due to the choice of $\psi^2$, the displacement is controlled towards the desired deformation $\bu_{T,d}$.
Thus, choosing $\psi^2$ and $\Uad^2$ allows the control of the springback of the solid body.
This is of great interest in applications, e.g.\ deep-drawing of metal sheets.

\short{}{Let us show the existence of an optimal control. }%
Using standard arguments, see, e.g.\ \cite[Theorem~4.15]{Troeltzsch2010:1}, the existence of a global minimizer of \eqref{eq:continuous_ulp} is a straightforward consequence of the weak continuity of $\GG\circ E$ proven in \autoref{thm:weak_continuity_solution_operator}.
\short{The proof is contained in the extended version of this paper, see \cite{Wachsmuth2011:1_extended}.}{}

\begin{theorem}
	There exists a global minimizer of \eqref{eq:continuous_ulp}.
	\label{thm:existence_continuous}
\end{theorem}
\short{}{
	\begin{proof}
		In virtue of the control-to-state map $\GG \circ E$, we define the reduced objective $f(\bg) := F(\GG^\bu(E\bg),\bg)$.
		Due to \autoref{asm:psi_lsc}, $f$ is bounded from below, hence $\inf f$ on $\Uad$ exists.
		We define
		$j := \inf \{ f(\bg), \bg \in \Uad\} \in \R$.

		Let $\{\bg_k\} \subset \Uad$ be a minimizing sequence, i.e.\ $\lim f(\bg_k) = j$.
		Since $\psi$ is bounded from below and $\nu > 0$, the sequence $\{\bg_k\}$ is bounded in $H^1(0,T;U)$.
		Hence, there is a weakly convergent subsequence, denoted by the same symbol. We denote the weak limit by $\bg$.
		The closeness and convexity of $\Uad$ implies the weak closeness.
		Therefore, $\{\bg_k\} \subset \Uad$ implies $\bg \in \Uad$.

		Since $\GG^\bu \circ E$ is weakly continuous from $H_{\{0\}}^1(0,T;U)$ to $H_{\{0\}}^1(0,T;V)$, see \autoref{thm:weak_continuity_solution_operator},
		the weak lower semicontinuity of $\psi$ and $\norm{\cdot}_{H^1(0,T;U)}^2$ implies that
		the reduced objective 
		\begin{equation*}
			f
			=
			F( \GG^\bu \circ E(\cdot), \cdot )
			=
			\psi(\GG^\bu\circ E(\cdot)) + \frac\nu2 \, \norm{\cdot}_{H^1(0,T;U)}^2
		\end{equation*}
		is weakly lower semicontinuous.
		Therefore,
		\begin{align*}
			j & = \lim f(\bg_k)      &  & \text{($\bg_k$ is a minimizing sequence)}\\
			& \ge f(\bg)           &  & \text{(weak lower semicontinuity)}\\
			& \ge j                &  & \text{($j = \inf f$ and $\bg \in \Uad$).}
		\end{align*}
		This shows $F(\GG^\bu(E\bg), \bg) = f(\bg) = j$, hence $\bg$ is a global minimizer.
	\end{proof}
}

\section{Time discretization}
\label{sec:time_discretization}

In this section we study a discretization in time of the optimal control problem \eqref{eq:continuous_ulp}.
The time-discrete version of the forward system \eqref{eq:VI_lower_level_introduction} is introduced in \autoref{subsec:introduction_time_discretization}.
The strong convergence of the time-discrete solutions $(\bSigma^\tau,\bu^\tau)$ is shown in \autoref{subsec:strong_convergence_forward}.
In particular, we prove a new rate of convergence of $(\bSigma^\tau,\bu^\tau)$ in $L^\infty(0,T;S^2 \times V)$ without assuming additional regularity of $\bSigma$, see \autoref{lem:sigma_tau_to_sigma_in_L_infty}.
Moreover, we show the convergence of $(\bSigma^\tau,\bu^\tau)$ in $H^1(0,T;S^2\times V)$ (by using an idea of \cite{Krejci1998}) and the convergence of $\lambda^\tau$ in $L^2(0,T;L^2(\Omega))$,
see \hyperref[thm:strong_convergence]{\theoremautorefname{}s~\ref*{thm:strong_convergence}} and \ref{thm:strong_convergence_lambda}, respectively.
These strong convergence results are new for the analysis of the quasistatic problem \eqref{eq:VI_lower_level_introduction}
and both
are essential for passing to the limit in the optimality system in \autoref{Wachsmuth2011:4:sec:weak_stationarity_quasistatic}.

In \autoref{sec:weak_convergence_of_time_discretization} we prove the weak convergence of $(\bSigma^\tau,\bu^\tau)$ in $H^1(0,T;S^2 \times V)$ assuming the weak convergence of the controls $\bg^\tau$ in $H^1(0,T;U)$.
This result is necessary in order to show in \autoref{subsec:approx_by_time_discrete} that \emph{every} local minimizer of \eqref{eq:continuous_ulp} can be approximated by local minimizers of (slightly modified) time-discrete problems \hyperref[eq:time-discrete_ulp_mod]{\textup{($\mathbf{P}_\bg^\tau$)}}, see \autoref{thm:continuous_approximation_with_local_solutions_mod}.
This approximability of local minimizers is essential for the derivation of \emph{necessary} optimality conditions which are satisfied for \emph{every} local minimizer of \eqref{eq:continuous_ulp} in \citeautoref{Wachsmuth2011:4:sec:weak_stationarity_quasistatic}.

\subsubsection*{Notation}
Throughout the paper
we partition the time horizon $[0,T]$ into $N$ intervals, each of constant length $\tau = T/N$ for simplicity.
We use a superscript $\tau$, i.e.\ $(\cdot)^\tau$, to indicate variables and operators associated with the discretization in time.

For a time-discrete variable $f^\tau \in X^N$, where $X$ is some Banach space, the components of $f^\tau$ are denoted by $f\taui$, $i = 1,\ldots,N$.
When necessary, we may refer to $f^\tau_0$ with a pre-defined value (interpreted as an initial condition), mostly $f^\tau_0 = 0$.

If $f^\tau \in X^N$ is the time-discretization of a variable $f \in H_{\{0\}}^1(0, T; X)$, we identify $f^\tau$ with its \emph{linear} interpolation, i.e., for $t \in [(i-1)\,\tau, i\,\tau]$ we define
\begin{equation}
	\label{eq:linear_interpolation}
	f^\tau(t) = \frac{t - (i-1)\,\tau}{\tau} \, f\taui + \frac{i\,\tau - t}{\tau} \, f\tauim,
\end{equation}
with $f^\tau_0 = 0$. Therefore, $X^N$ is identified with a subspace of $H_{\{0\}}^1(0,T; X)$.
We make use of this identification for the variables $\bSigma$, $\bu$ and $\bg$.
For later reference, we remark
\begin{equation}
	\dot f^\tau(t) = \frac{1}{\tau} \, (f\taui - f\tauim)
	\quad\text{and}\quad
	\label{eq:linear_interpolation_vs_constant}
	f\taui
	=
	f^\tau(t) - ( t - i \, \tau ) \, \dot f^\tau(t)
\end{equation}
for almost all $t \in ((i-1)\,\tau, i\,\tau)$ and $i = 1,\ldots,N$.

On the other hand, if $f^\tau \in X^N$ is the discretization in time of a variable belonging to $L^2(0, T; X)$, we identify $f^\tau$ with a piecewise \emph{constant} function in $L^2(0,T;X)$, i.e., for $t \in [(i-1)\,\tau, i\,\tau)$ we define
\begin{equation}
	\label{eq:constant_interpolation}
	f^\tau(t) = f\taui.
\end{equation}
This is used for the plastic multiplier $\lambda$.

\subsection{Introduction of a discretization in time of the forward problem}
\label{subsec:introduction_time_discretization}
In this section we introduce a discretization in time of the forward problem \eqref{eq:VI_lower_level_introduction}.
Replacing the time derivatives by backward differences, we obtain the discretized problem:
given $\ell^\tau \in (V')^N$, find $(\bSigma^\tau, \bu^\tau) \in (S^2 \times V)^N$ such that $\bSigma\taui \in \KK$
and
\begin{subequations}
	\label{eq:Lower-Level_Problem_semidiscretized}
	\begin{align}
		\dual{ A (\bSigma\taui- \bSigma\tauim) + B^\star (\bu\taui-\bu\tauim) }{ \bT - \bSigma\taui } &\ge \mrep{0}{\ell\taui} \quad \text{for all }\bT \in \KK, \label{eq:Lower-Level_Problem_semidiscretized1} \\
		B\bSigma\taui &= \ell\taui \quad \text{in } V', \label{eq:Lower-Level_Problem_semidiscretized2}
	\end{align}
\end{subequations}
holds for all $i \in \{1,\ldots,N\}$,
where $(\bSigma^\tau_0,\bu^\tau_0) = \bnull$.
We denote the solution operator which maps $\ell^\tau \to (\bSigma^\tau, \bu^\tau)$ by $\GG^\tau$.
Let us mention that, for fixed $i$, \eqref{eq:Lower-Level_Problem_semidiscretized} can be interpreted as the solution of a \emph{static} plasticity problem.
Several properties of this time discretization are proven in \cite[Proof of Theorem~8.12, page~196]{HanReddy1999}.
We recall the results which are important for the following analysis.

In order to show an analog to \eqref{eq:relation_sigma_chi_u} for the time-discrete problem, we test \eqref{eq:Lower-Level_Problem_semidiscretized1} with $\bT = \bSigma\taui + (\btau, -\btau)$, where $\btau \in S$ is arbitrary. Using $(\bSigma^\tau_0, \bu^\tau_0) = \bnull$ we obtain
\begin{equation}
	\label{eq:relation_sigma_chi_u_tau}
	\C^{-1} \bsigma^\tau - \bvarepsilon(\bu^\tau) - \H^{-1} \bchi^\tau = \bnull \quad\text{a.e.\ in } (0,T) \times \Omega.
\end{equation}
In \cite[Lemma~8.8 and (8.37)]{HanReddy1999}, we find the a-priori estimate
\begin{equation}
	\label{eq:a-priori_time_bsigma}
	\norm{\bSigma^\tau}_{H^1(0,T; S^2)} \le C \, \norm{\ell^\tau}_{H^1(0,T; V')}.
\end{equation}
Together with the inf-sup condition \eqref{eq:inf-sup} of $B^\star$ and \eqref{eq:relation_sigma_chi_u_tau}, we obtain
\begin{equation}
	\label{eq:a-priori_time}
	\norm{\bSigma^\tau}_{H^1(0,T; S^2)} + \norm{\bu^\tau}_{H^1(0,T; V)} \le C \, \norm{\ell^\tau}_{H^1(0,T; V')}.
\end{equation}
Due to this a-priori estimate, \cite{HanReddy1999} are able to prove the weak convergence of $(\bSigma^\tau,\bu^\tau) = \GG^\tau(\ell^\tau)$
towards the solution $(\bSigma,\bu) = \GG(\ell)$ of \eqref{eq:Lower-Level_Problem} in $H^1(0,T;S^2\times V)$
if the data $\ell\taui$ in \eqref{eq:Lower-Level_Problem_semidiscretized} is chosen as the point evaluation of $\ell$.

From an optimization point of view, this result is too weak for two reasons.
First, the convergence is not w.r.t.\ the strong topology in $H^1(0,T;S^2\times V)$.
Second, the convergence is proven for a restricted choice of $\ell^\tau$.
We therefore strengthen the results of \cite{HanReddy1999} in \autoref{thm:strong_convergence}.

As for the continuous problem, the time-discrete problem can be formulated as a complementarity system.
The results in \citeautoref{HerzogMeyerWachsmuth2010:1:sec:existence} prove the existence of $\lambda^\tau \in L^2(\Omega)^N$ such that
the complementarity system
\begin{subequations}
	\label{eq:Lower-Level_Problem_semidiscretized_multi}
	\begin{align}
		A (\bSigma\taui-\bSigma\tauim) + B^\star (\bu\taui-\bu\tauim) + \tau \, \lambda\taui \, \DD^\star \DD \bSigma\taui &= \mrep{\bnull}{\ell\taui} \quad \text{in } S^2, \label{eq:Lower-Level_Problem_semidiscretized_multi1} \\
		B\bSigma\taui &= \ell\taui \quad \text{in } V', \label{eq:Lower-Level_Problem_semidiscretized_multi2}\\
		0 \le \lambda\taui \quad \perp \quad \phi(\bSigma\taui) & \le \mrep{0}{\ell\taui} \quad \text{a.e.\ in } \Omega \label{eq:Lower-Level_Problem_semidiscretized_multi3}
	\end{align}
\end{subequations}
is satisfied.
Here, we scaled the multiplier $\lambda^\tau$ appropriately.

Similarly to \eqref{eq:relation_sigma_chi_u_tau} we obtain
a representation formula for $\lambda^\tau$,
\begin{equation}
	\label{eq:-h_chi=lambda_DD_bsigma_discrete}
	-\C^{-1} (\bsigma\taui - \bsigma\tauim) + \bvarepsilon( \bu\taui - \bu\tauim ) = -\H^{-1} (\bchi\taui - \bchi\tauim) = \tau \, \lambda\taui \, \DD\bSigma\taui.
\end{equation}

\begin{remark}
	\label{rem:same_time_discretization_as_krejci}
	The same time discretization approach is used in \cite{Krejci1996,Krejci1998} to prove the existence of a solution of \eqref{eq:abstract_evi}.
	Indeed, by applying the time discretization \cite[(3.6)]{Krejci1998}
	to the reduced formulation \eqref{eq:eq:Lower-Level_Problem_evolution_vi}, we obtain: find $\bSigma\taui \in \KK_{\ell\taui}$, such that
	\begin{equation*}
		\dual{ A (\bSigma\taui- \bSigma\tauim) }{ \bT - \bSigma\taui } \ge 0 \quad \text{for all }\bT \in \KK_{\ell\taui}.
	\end{equation*}
	As discussed for the continuous problem \eqref{eq:VI}, we can introduce a Lagrange multiplier associated with $B\bSigma\taui = \ell\taui$ such that $(\bSigma^\tau, \bu^\tau)$ satisfy \eqref{eq:Lower-Level_Problem_semidiscretized}.
\end{remark}

In the following two subsections we prove that the discretization in time converges in a strong and a weak sense,
similar to the continuity properties of the forward operator $\GG$ proven in \hyperref[sec:evolution_vi]{\subsectionautorefname{}s~\ref*{sec:evolution_vi}} and \ref{subsec:weak_continuity}.
Under the assumption that $\ell^\tau$ converges strongly to $\ell$ in the space $H^1(0,T;V')$ or that $\bg^\tau$ converges weakly in the space $H^1(0,T;U)$,
the corresponding solutions $(\bSigma^\tau, \bu^\tau)$ are shown to converge strongly (or weakly) to $(\bSigma, \bu)$ in $H^1(0,T;S^2\times V)$, respectively.

Both statements will be needed in \autoref{subsec:approx_by_time_discrete} to show the approximability of local solutions of \eqref{eq:continuous_ulp} mentioned in the introduction of this section.

\subsection{Strong convergence}
\label{subsec:strong_convergence_forward}
The strong convergence of the discretization in time of an EVI can be found in \cite[Proposition~I.3.11]{Krejci1996}.
He proves the convergence in $W^{1,1}(0,T;X)$ of the time-discrete solutions.
The data of the time-discretized problems are taken as the point evaluations of the data of the continuous problem.
We use these ideas to prove the convergence of $(\bSigma^\tau, \bu^\tau)$ in $H^1(0,T; S^2 \times V)$
whenever the right-hand sides $\ell^\tau$ converge in $H^1(0,T;V')$ as $\tau \searrow 0$.
Moreover, we show the convergence of $\lambda^\tau$ in $L^2(0,T;L^2(\Omega))$ as $\tau \searrow 0$.
This is a new result for the analysis of \eqref{eq:VI_lower_level_introduction} and a crucial ingredient for proving the optimality system in \citeautoref{Wachsmuth2011:4:sec:weak_stationarity_quasistatic}.

Let us consider a sequence of time steps $\{\tau_k\}$ and a sequence of loads $\ell^{\tau_k} \in (V')^{N_k}$, with $N_k \, \tau_k = T$, such that the linear interpolants $\ell^{\tau_k} \in H_{\{0\}}^1(0,T; V')$ converge in $H^1(0,T; V')$ towards some $\ell \in H_{\{0\}}^1(0,T; V')$.
We denote by $(\bSigma^{\tau_k}, \bu^{\tau_k}) \in (S^2 \times V)^{N_k}$ the solutions of the time-discrete problem \eqref{eq:Lower-Level_Problem_semidiscretized1}, as well as their piecewise linear interpolants.
The aim of this section is to prove $(\bSigma^{\tau_k}, \bu^{\tau_k}) \to (\bSigma, \bu)$ in $H^1(0,T; S^2 \times V)$, where $(\bSigma, \bu)$ is the solution to the continuous problem \eqref{eq:Lower-Level_Problem}.

For simplicity of notation, we omit the index $k$ and we refer to ``$\ell^{\tau_k} \to \ell$ if $k \to \infty$'' simply as ``$\ell^\tau \to \ell$ as $\tau \searrow 0$''.

The proof of \autoref{thm:strong_convergence} relies on the convergence argument \autoref{cor:convergence_in_W^1,p}
which is tailored to the analysis of EVIs.
Its prerequisite \eqref{eq:condition_of_cor:convergence_in_W^1,p_1} is the 
convergence of $\bSigma^\tau$ in $L^\infty(0,T; S^2)$.
This is addressed in a preliminary step and required in the proof of \autoref{thm:strong_convergence}.
\begin{lemma}
	\label{lem:sigma_tau_to_sigma_in_L_infty}
	Let $\ell^\tau$ be a bounded sequence in $H^1_{\{0\}}(0,T; V')$ such that
	$\ell^\tau \to \ell$ in $W^{1,1}(0,T; V')$ as $\tau \searrow 0$. Then
	\begin{equation}
		\label{eq:sigma_tau_to_sigma_in_L_infty}
		\norm{\bSigma - \bSigma^\tau}_{L^\infty(0,T; S^2)}^2
		\le C \, \big(\norm{\dot\ell - \dot\ell^\tau}_{L^1(0,T; V')}^2 + \tau \big).
	\end{equation}
	In particular, $\bSigma^\tau \to \bSigma$ in $L^\infty(0,T;S^2)$ as $\tau \searrow 0$.
\end{lemma}
\begin{proof}
	We have
	\begin{equation}
		\label{eq:time_derivative_of_norm}
		\frac12 \frac{\d}{\d t} \norm{\bSigma(t) - \bSigma^\tau(t)}_A^2
		=
		\dual{A\dot\bSigma(t) - A\dot\bSigma^\tau(t)}{\bSigma(t) - \bSigma^\tau(t)} \quad\text{f.a.a.\ } t \in (0,T).
	\end{equation}
	Let us fix some $t \in ((i-1)\,\tau, i\,\tau)$, $i \in \{1, \ldots, N\}$.
	Since $\bSigma^\tau(t)$ is a convex combination of $\bSigma\taui$ and $\bSigma\tauim$, see \eqref{eq:linear_interpolation},  and $\bSigma\tauim, \bSigma\taui \in \KK$, we obtain $\bSigma^\tau(t) \in \KK$.
	Hence, 
	$\bT = \bSigma^\tau(t) + \bSigma_{\ell-\ell^\tau}(t) \in \KK$.
	Using $B\bT = \ell(t)$, \eqref{eq:Lower-Level_Problem1} implies
	\begin{equation}
		\label{eq:continuous_estimate}
		\bigdual{ A \dot\bSigma(t)}{\bSigma^\tau(t)-\bSigma(t)
			+
			\bSigma_{\ell-\ell^\tau}(t)
		}
		\ge 0.
	\end{equation}
	Testing \eqref{eq:Lower-Level_Problem_semidiscretized1} with
	$\bT = \bSigma(t) + \bSigma_{\ell\taui - \ell(t)}$ yields
	\begin{equation*}
		\bigdual{ A (\bSigma\taui-\bSigma\tauim)}{\bSigma(t) - \bSigma\taui
			+
			\bSigma_{\ell\taui - \ell(t)}
		}
		\ge 0.
	\end{equation*}
	Using \eqref{eq:linear_interpolation_vs_constant}
	we obtain
	\begin{equation*}
		\bigdual{ A \dot\bSigma^\tau(t)}{\bSigma(t) - \bSigma^\tau(t) - (i\,\tau-t) \, \dot\bSigma^\tau
			+
			\bSigma_{\ell^\tau(t) - \ell(t) + (i\,\tau-t)\,\dot\ell^\tau(t)}
		}
		\ge 0.
	\end{equation*}
	Together with \eqref{eq:continuous_estimate} this yields
	\begin{align*}
		\dual{A\dot\bSigma(t) - A\dot\bSigma^\tau(t)}{\bSigma(t) - \bSigma^\tau(t)}
		&\le
		\bigdual{A\dot\bSigma(t) - A\dot\bSigma^\tau(t)}{
			\bSigma_{\ell-\ell^\tau}(t)
		}
		\\
		&\qquad
		+
		(i\,\tau-t) \, 
		\bigdual{ A \dot\bSigma^\tau(t)}{-\dot\bSigma^\tau(t)
			+
			\bSigma_{\dot\ell^\tau}(t)
		}.
	\end{align*}
	Integrating over $t$ and using \eqref{eq:time_derivative_of_norm}, together with the boundedness of $\dot\bSigma^\tau$ and $\bSigma_{\dot\ell^\tau}$ in $L^2(0,T; S^2)$, see \eqref{eq:a-priori_time_bsigma}, yields
	\begin{align*}
		\frac12 \norm{\bSigma(t) - \bSigma^\tau(t)}_A^2
		\le
		\int_0^t 
		\bigdual{A\dot\bSigma(s) - A\dot\bSigma^\tau(s)}{ \bSigma_{\ell-\ell^\tau}(s) }
		\, \d s
		+
		C \, \tau.
	\end{align*}
	Integrating by parts and using $\bSigma(0) = \bSigma^\tau(0) = \bnull$ gives
	\begin{align*}
		\frac12 \norm{\bSigma(t) - \bSigma^\tau(t)}_A^2
		&\le
		\bigdual{A\bSigma(t) - A\bSigma^\tau(t)}{\bSigma_{\ell-\ell^\tau}(t)} \\
		&\qquad-
		\int_0^t 
		\bigdual{A\bSigma(s) - A\bSigma^\tau(s)}{ \bSigma_{\dot\ell-\dot\ell^\tau}(s) }
		\, \d s
		+
		C \, \tau\\
		& \le
		C \, \big(\norm{\bSigma(t) - \bSigma^\tau(t)}_{S^2} \, \norm{\ell(t) - \ell^\tau(t)}_{V'} \\
			& \qquad + \norm{\bSigma - \bSigma^\tau}_{L^\infty(0,T; S^2)} \, \norm{\dot\ell - \dot\ell^\tau}_{L^1(0,T; V')}
		+ \tau \big).
	\end{align*}
	Taking the supremum over $t \in (0,T)$ on both sides shows
	\begin{align*}
		\norm{\bSigma - \bSigma^\tau}_{L^\infty(0,T; S^2)}^2
		& \le
		C \, \big(\norm{\bSigma - \bSigma^\tau}_{L^\infty(0,T; S^2)} \, \norm{\ell - \ell^\tau}_{L^\infty(0,T;V')} \\
			& \qquad + \norm{\bSigma - \bSigma^\tau}_{L^\infty(0,T; S^2)} \, \norm{\dot\ell - \dot\ell^\tau}_{L^1(0,T; V')}
		+ \tau \big).
	\end{align*}
	Finally, Young's inequality and $W^{1,1}(0,T; V') \embeds L^\infty(0,T; V')$ yields
	\begin{equation*}
		\norm{\bSigma - \bSigma^\tau}_{L^\infty(0,T; S^2)}^2
		\le C \, \big(\norm{\dot\ell - \dot\ell^\tau}_{L^1(0,T; V')}^2 + \tau \big).
	\end{equation*}
\end{proof}
Let us comment on the case when $\ell^\tau$ is the point evaluation of $\ell$.
The estimate \eqref{eq:sigma_tau_to_sigma_in_L_infty}
shows the order of convergence of $\tau^{1/2}$ w.r.t.\ the $L^\infty$-norm provided that $\dot\ell \in W^{1,1}(0,T;V')$.
This rate of convergence is a new result for the time-discretization of \eqref{eq:VI_lower_level_introduction}.

In \cite[Theorem~13.1]{HanReddy1999} the estimate
\begin{equation*}
	\norm{\bSigma - \bSigma^\tau}_{L^\infty(0,T;S^2)}
	\le
	c \, \tau \, \norm{\bSigma}_{W^{2,1}(0,T;S^2)}
\end{equation*}
is shown under the assumption $\bSigma \in W^{2,1}(0,T;S^2)$.
However,
this assumption is unlikely to hold for the solution of an EVI.
There are many examples where the solution of an EVI only belongs to $W^{1,\infty}(0,T;X)$, even if $X$ is one-dimensional and the input is smooth, see e.g.\ \cite[Examples~1, 2]{KrejciLovicar1990}.
Hence, $\bSigma \in W^{2,1}(0,T;S^2)$ cannot be considered the generic case.

\begin{theorem}
	\label{thm:strong_convergence}
	Let us assume $\ell^\tau \to \ell$ in $H_{\{0\}}^1(0,T; V')$ as $\tau \searrow 0$.
	Then $(\bSigma^\tau, \bu^\tau) \to (\bSigma, \bu)$ in $H_{\{0\}}^1(0,T; S^2 \times V)$ as $\tau \searrow 0$.
\end{theorem}
\begin{proof}
	We follow the idea of \cite[Proof of Theorem~3.6]{Krejci1998}.
	We will apply \autoref{cor:convergence_in_W^1,p} with the setting
	\begin{equation}
		\label{eq:setting_for_cor:convergence_in_W^1,p}
		\begin{aligned}
			X &= S^2_A, &
			u_n &= \bSigma_{\ell^\tau} - 2 \, \bSigma^\tau, &
			u_0 &= \bSigma_\ell - 2 \, \bSigma, \\
			p &= 2, &
			g_n &= \norm{\bSigma_{\dot\ell^\tau}(\cdot)}_A, &
			g_0 &= \norm{\bSigma_{\dot\ell}(\cdot)}_A.
		\end{aligned}
	\end{equation}
	to show the convergence of $\bSigma^\tau$.
	Here, $S^2_A$ denotes the Hilbert space $S^2$ equipped with the inner product induced by $A$.
	To apply \autoref{cor:convergence_in_W^1,p}, we have to verify the prerequisites \eqref{eq:condition_of_cor:convergence_in_W^1,p_3} and \eqref{eq:condition_of_cor:convergence_in_W^1,p_4}.

	By definition of the discrete problem \eqref{eq:Lower-Level_Problem_semidiscretized}, we have for a fixed $i \in \{ 1, 2, \ldots, N \}$
	\begin{equation*}
		\dual{ A( \bSigma_i^\tau - \bSigma_{i-1}^\tau) }{ \bT - \bSigma_i^\tau} \ge 0 \quad\text{for all } \bT \in \KK,\; B \bT = \ell\taui,
	\end{equation*}
	which is the time-discrete analog to \eqref{eq:VI}.
	We choose $\bT = \bSigma_{i-1}^\tau + \bSigma_{\ell\taui - \ell\tauim}$, see \eqref{eq:bSigma(ell)} for the definition of $\bSigma_\ell$.
	By definition of $\bSigma_\ell$ we have $B \bT  = \ell_i^\tau$ and the shift invariance \eqref{eq:shift-invariance} implies $\bT \in \KK$.
	This shows
	\begin{equation*}
		\Bigdual{ A( \bSigma_i^\tau - \bSigma_{i-1}^\tau) }{ \bSigma_{i-1}^\tau + \bSigma_{\ell\taui - \ell\tauim} - \bSigma_i^\tau} \ge 0.
	\end{equation*}
	Therefore dividing by $\tau^2$ and using the notation of linear interpolants, see \eqref{eq:linear_interpolation}, implies
	\begin{equation*}
		\Bigdual{ A \dot{\bSigma}^\tau(t) }{ \bSigma_{\dot\ell^\tau}(t) - \dot{\bSigma}^\tau(t)} \ge 0 \quad\text{f.a.a.\ } t \in (0,T).
	\end{equation*}
	This shows
	\begin{equation}
		\label{eq:thm:strong_convergence_1}
		\norm{\bSigma_{\dot\ell^\tau}(t) - 2 \, \dot{\bSigma}^\tau(t)}_A \le \norm{ \bSigma_{\dot\ell^\tau}(t)}_A \quad\text{f.a.a.\ } t \in (0,T),
	\end{equation}
	where $\norm{\cdot}_A$ is the norm on $S^2$ induced by $A$, see \eqref{eq:norm_A}.

	To derive a similar formula for the continuous solution, we test \eqref{eq:VI} with $\bT = \bSigma(t+h) + \bSigma_{\ell(t)-\ell(t+h)}$ for $h >0$ and obtain
	\begin{equation*}
		\Bigdual{A\dot\bSigma(t)}{
			\bSigma_{\ell(t)-\ell(t+h)}
			+
			\bSigma(t+h)-\bSigma(t)
		}
		\ge 0.
	\end{equation*}
	Passing to the limit $h \searrow 0$, using \cite[Theorem~8.14]{Krejci1998}, yields
	\begin{equation*}
		\Bigdual{A\dot\bSigma(t)}{
			\bSigma_{\dot\ell}(t)
			-
			\dot\bSigma(t)
		}
		\le 0\quad\text{f.a.a.\ } t \in (0,T).
	\end{equation*}
	Analogously, we obtain by $h \nearrow 0$
	\begin{equation*}
		\Bigdual{A\dot\bSigma(t)}{
			\bSigma_{\dot\ell}(t)
			-
			\dot\bSigma(t)
		}
		\ge 0\quad\text{f.a.a.\ } t \in (0,T).
	\end{equation*}
	Hence
	\begin{equation}
		\label{eq:thm:strong_convergence_2}
		\norm{\bSigma_{\dot\ell}(t) - 2\,\dot{\bSigma}(t)}_A = \norm{\bSigma_{\dot\ell}(t)}_A\quad\text{f.a.a.\ } t \in (0,T).
	\end{equation}
	In order to apply \autoref{cor:convergence_in_W^1,p} with the setting \eqref{eq:setting_for_cor:convergence_in_W^1,p}
	we check the prerequisites.
	\begin{itemize}
		\item[\eqref{eq:condition_of_cor:convergence_in_W^1,p_1}:]
			\autoref{lem:sigma_tau_to_sigma_in_L_infty} ensures $\bSigma^\tau \to \bSigma$ in $L^\infty(0,T; S^2_A)$ and the assumption $\ell^\tau \to \ell$ in $H^1(0,T;V')$ implies $\bSigma_{\ell^\tau} \to \bSigma_\ell$ in $L^\infty(0,T;S^2)$.
		\item[\eqref{eq:condition_of_cor:convergence_in_W^1,p_2}:]
			By the linearity of $\ell \mapsto \bSigma_\ell$,
			the assumption $\ell^\tau \to \ell$ in $H^1(0,T; V')$ implies $\norm{\bSigma_{\dot\ell^\tau}(\cdot)}_A \to \norm{\bSigma_{\dot\ell}(\cdot)}_A$ in $L^2(0,T; \R)$.
		\item[\eqref{eq:condition_of_cor:convergence_in_W^1,p_3}:]
			This was shown in \eqref{eq:thm:strong_convergence_1}.
		\item[\eqref{eq:condition_of_cor:convergence_in_W^1,p_4}:]
			This was shown in \eqref{eq:thm:strong_convergence_2}.
	\end{itemize}
	Therefore, \autoref{cor:convergence_in_W^1,p} yields
	\begin{equation*}
		\norm{\bSigma_{\ell^\tau} - 2 \, \bSigma^\tau - (\bSigma_\ell - 2 \, \bSigma)}_{H^1(0,T;S^2_A)} \to 0.
	\end{equation*}
	By the assumption $\ell^\tau \to \ell$ in $H^1(0,T;V')$ we infer
	\begin{equation*}
		\norm{\bSigma^\tau - \bSigma}_{H^1(0,T; S^2_A)} \to 0 \quad\text{as } \tau \searrow 0.
	\end{equation*}

	Using \eqref{eq:relation_sigma_chi_u}, \eqref{eq:relation_sigma_chi_u_tau} and the inf-sup condition of $B$, we obtain $\bu^\tau \to \bu$ in $H^1(0,T; V)$ as $\tau \searrow 0$.
\end{proof}
Using the representation formula \eqref{eq:-h_chi=lambda_DD_bsigma_discrete} for $\lambda^\tau$ we are able to prove the strong convergence of $\lambda^\tau$ in $L^2(0,T;L^2(\Omega))$ towards $\lambda$.
We remark that this is a novel result for the analysis of \eqref{eq:Lower-Level_Problem_multi}.

Let us mention that this strong convergence is crucial for the analysis in \cite{Wachsmuth2011:4}.
Without this strong convergence,
we could neither pass to the limit in the complementarity relation $\mu^\tau \, \lambda^\tau = 0$
nor in the adjoint equation, see \citeautoref{Wachsmuth2011:4:sec:weak_stationarity_quasistatic}.
Hence, this convergence is essential to prove the necessity of the system of weakly stationary type in \cite{Wachsmuth2011:4}.
\begin{theorem}
	\label{thm:strong_convergence_lambda}
	Let $\ell^\tau \to \ell$ in $H_{\{0\}}^1(0,T; V')$ as $\tau \searrow 0$. Then $\lambda^\tau \to \lambda$ in $L^2(0,T; L^2(\Omega))$ as $\tau \searrow 0$.
\end{theorem}
\begin{proof}
	\textit{Step~(1):}
	We show that $\lambda^\tau \to \lambda$ in $L^1(0,T;L^1(\Omega))$.
	Using $\lambda\taui \, \abs{\DD\bSigma\taui}^2 = \tilde\sigma_0^2 \, \lambda\taui$ by \eqref{eq:Lower-Level_Problem_semidiscretized_multi3} and \eqref{eq:-h_chi=lambda_DD_bsigma_discrete}, we obtain
	\begin{equation*}
		\lambda\taui
		= \frac{1}{\tilde\sigma_0^2} \, \lambda\taui \, \DD\bSigma\taui \dprod \DD\bSigma\taui
		= \frac{1}{\tau \, \tilde\sigma_0^2} (-\H^{-1} (\bchi\taui-\bchi\tauim ) \dprod \DD\bSigma\taui).
	\end{equation*}
	Using the notion of piecewise linear interpolations $\bchi^\tau$, $\bSigma^\tau$ and the piecewise constant interpolation $\lambda^\tau$, we obtain for $t \in ((i-1)\,\tau, i\,\tau)$, see \eqref{eq:linear_interpolation_vs_constant},
	\begin{equation*}
		\lambda^\tau(t) = \frac{-1}{\tilde\sigma_0^2} \Big(\H^{-1} \dot\bchi^\tau(t) \dprod \big[\DD\bSigma^\tau(t) - (t -i\,\tau)\,\DD\dot\bSigma^\tau(t)\big]\Big).
	\end{equation*}
	Using $\lambda = -\H^{-1} \dot\bchi \dprod \DD\bSigma / \tilde\sigma_0^2$, by \eqref{eq:Lower-Level_Problem_multi}, this implies
	\begin{align*}
		\tilde\sigma_0^2 \, \abs{\lambda^\tau(t) - \lambda(t)}
		&\le \Bigabs{ \H^{-1} \dot\bchi^\tau(t) \dprod \DD\bSigma^\tau(t) - \H^{-1}\dot\bchi(t) \dprod \DD\bSigma(t)} \\
		&\qquad + (t -i\,\tau)\, \Bigabs{ \H^{-1}\dot\bchi^\tau(t)\dprod\DD\dot\bSigma^\tau(t) }\\
		&\le \Bigabs{ \H^{-1} \dot\bchi^\tau(t) \dprod \DD\bSigma^\tau(t) - \H^{-1}\dot\bchi(t) \dprod \DD\bSigma(t)} \\
		&\qquad + \tau \, \Bigabs{ \H^{-1}\dot\bchi^\tau(t)\dprod\DD\dot\bSigma^\tau(t) }.
	\end{align*}
	Hence,
	\begin{align*}
		\norm{\lambda^\tau - \lambda}_{L^1(0,T;L^1(\Omega))}
		&\le c \, \Big(
			\norm{\dot\bchi^\tau} \, \norm{\DD\bSigma^\tau - \DD\bSigma} +
			\norm{\DD\bSigma} \, \norm{\dot\bchi^\tau - \dot\bchi}\\
			&\qquad + \tau \, \norm{\dot\bchi^\tau} \, \norm{\DD\dot\bSigma^\tau}
		\Big),
	\end{align*}
	where all norms on the right-hand side are those of $L^2(0,T;S)$.
	Using $\bSigma^\tau \to \bSigma$ in $H^1(0,T; S^2)$ implies $\lambda^\tau \to \lambda$ in $L^1(0,T; L^1(\Omega))$.

	\textit{Step~(2):}
	We show that $\norm{\lambda^\tau}_{L^2(0,T;L^2(\Omega))} \to \norm{\lambda}_{L^2(0,T;L^2(\Omega))}$.
	Using that $\lambda^\tau$ is the piecewise constant interpolation of $(\lambda\taui)_{i=1}^N$, we obtain
	\begin{equation*}
		\norm{\lambda^\tau}_{L^2(0,T;L^2(\Omega))}^2
		= \tau \sum_{i=1}^{N} \int_\Omega (\lambda\taui)^2 \, \d x.
	\end{equation*}
	Using $\lambda\taui \, \abs{\DD\bSigma\taui} = \tilde\sigma_0 \, \lambda\taui$ by \eqref{eq:Lower-Level_Problem_semidiscretized_multi3}, we obtain
	\begin{align*}
		\norm{\lambda^\tau}_{L^2(0,T;L^2(\Omega))}^2
		& = \frac{\tau}{\tilde\sigma_0^2} \sum_{i=1}^{N} \int_\Omega \abs{\lambda\taui \, \DD\bSigma\taui}^2 \, \d x \\
		& = \frac{\tau}{\tilde\sigma_0^2} \sum_{i=1}^{N} \int_\Omega \Bigabs{- \H^{-1}\frac{\bchi\taui-\bchi\tauim}{\tau}}^2 \, \d x && \text{(by \eqref{eq:-h_chi=lambda_DD_bsigma_discrete})} \\
		& = \frac{1}{\tilde\sigma_0^2}\norm{\H^{-1}\dot\bchi^\tau}_{L^2(0,T;S)}^2.
	\end{align*}
	An analogous calculation shows $\norm{\lambda}_{L^2(0,T; L^2(\Omega))} = \frac{1}{\tilde\sigma_0}\norm{\H^{-1}\dot\bchi}_{L^2(0,T;S)}$.
	Hence, by $\bSigma^\tau \to \bSigma$ in $H^1(0,T; S^2)$, see \autoref{thm:strong_convergence}, we infer the convergence of norms $\norm{\lambda^\tau}_{L^2(0,T; L^2(\Omega))} \to \norm{\lambda}_{L^2(0,T; L^2(\Omega))}$.

	\textit{Step~(3):}
	We show that $\lambda^\tau \to \lambda$ in $L^2(0,T;L^2(\Omega))$.
	By Step~(2) we know that for every subsequence of $\tau$, there is a subsequence $\tau_k$ such that $\lambda^{\tau_k} \weakly \tilde\lambda$ in $L^2(0,T;L^2(\Omega))$ for some $\tilde \lambda$.
	Step~(1) implies $\tilde\lambda = \lambda$.
	Hence, the whole sequence converges weakly. In view of the convergence of norms, $\lambda^\tau$ converges strongly to $\lambda$ in $L^2(0,T;L^2(\Omega))$ as $\tau\searrow0$.
\end{proof}
\subsection{Weak convergence}
\label{sec:weak_convergence_of_time_discretization}
In this section we show the weak convergence of the time-discrete states $(\bSigma^\tau, \bu^\tau) = \GG^\tau(E \bg^\tau)$ in $H_{\{0\}}^1(0,T;S^2 \times V)$,
under the assumption that the controls $\bg^\tau$ converges weakly towards $\bg$ in $H_{\{0\}}^1(0,T;U)$.
This result is essential for proving the approximability results in \autoref{subsec:approx_by_time_discrete}.
Similarly to the weak continuity of the forward operator proven in \autoref{subsec:weak_continuity}, we need the weak convergence of the right-hand sides $E \bg^\tau$ with respect to a stronger norm in space, see the discussion in the beginning of \autoref{subsec:weak_continuity}.
\begin{theorem}
	\label{thm:weak_convergence_of_time_discretization}
	Let $\bg^\tau \in U^n$ such that the linear interpolants converge weakly, i.e., there is $\bg \in H_{\{0\}}^1(0,T; U)$ such that $\bg^\tau \weakly \bg$ in $H_{\{0\}}^1(0, T; U)$ as $\tau \searrow 0$.

	Then the solutions $(\bSigma^\tau, \bu^\tau)$ of the time-discrete problem \eqref{eq:Lower-Level_Problem_semidiscretized} with right-hand side $\ell^\tau = E\bg^\tau$ converge weakly in $H_{\{0\}}^1(0, T; S^2 \times V)$ towards the solution $(\bSigma, \bu)$ of \eqref{eq:Lower-Level_Problem} with right-hand side $\ell = E\bg$ as $\tau \searrow 0$.
\end{theorem}
\begin{proof}
	Due to the a-priori bound \eqref{eq:a-priori_time}, there exists a weakly convergent subsequence of $(\bSigma^\tau, \bu^\tau)$, denoted by the same symbol, i.e.\ $(\bSigma^\tau, \bu^\tau) \weakly (\widetilde\bSigma, \widetilde\bu)$ in $H^1(0,T;S^2 \times V)$ as $\tau \searrow 0$.
	As in the proof of \autoref{thm:weak_continuity_solution_operator}, we shall show $(\widetilde\bSigma, \widetilde\bu) = (\bSigma, \bu)$. Hence the whole sequence converges weakly and this simplification of notation is justified.

	Let $\bT \in L^2(0, T; S^2)$ with $\bT \in \KK$ a.e.\ in $(0,T)$ be arbitrary.
	Let $i \in \{0, \ldots, N-1\}$ and $\kappa \in (0,1)$ be given.
	Set $t = (i + \kappa ) \, \tau$.
	Testing \eqref{eq:Lower-Level_Problem_semidiscretized1} with $\bT(t)$ yields
	\begin{equation*}
		\dual{A \dot\bSigma^\tau(t) + B^\star \dot\bu^\tau(t)}{\bT(t) - \bSigma^\tau( (i+1) \, \tau)} \ge 0.
	\end{equation*}
	Using
	$
	\bSigma^\tau( (i+1) \, \tau)
	=
	\bSigma^\tau(t) + ( 1 - \kappa ) \, \tau \, \dot\bSigma^\tau(t) 
	$
	by \eqref{eq:linear_interpolation},
	and integrating over $\kappa \in (0,1)$, i.e.\ $t \in (i\,\tau, (i+1)\,\tau)$ yields
	\begin{align*}
		0 
		&\le
		\int_{i\,\tau}^{(i+1)\,\tau}
		\dual{A \dot\bSigma^\tau(t) + B^\star \dot\bu^\tau(t)}{\bT(t) - \bSigma^\tau( (i+1) \, \tau)}
		\, \d t \\
		&=
		\int_{i\,\tau}^{(i+1)\,\tau} \dual{A \dot\bSigma^\tau + B^\star \dot\bu^\tau}{\bT - \bSigma^\tau} \, \d t
		-
		\frac{\tau}{2}
		\int_{i\,\tau}^{(i+1)\,\tau}
		\dual{A \dot\bSigma^\tau + B^\star \dot\bu^\tau}{\dot\bSigma^\tau} \, \d t.
	\end{align*}
	Hence, summing over $i = 0,\ldots, N-1$ implies
	\begin{equation}
		\label{eq:in_thm:weak_convergence_of_time_discretization}
		0 
		\le
		\int_{0}^{T} \dual{A \dot\bSigma^\tau + B^\star \dot\bu^\tau}{\bT - \bSigma^\tau} \, \d t
		-
		\frac{\tau}{2}
		\int_{0}^{T}
		\dual{A \dot\bSigma^\tau + B^\star \dot\bu^\tau}{\dot\bSigma^\tau} \, \d t.
	\end{equation}
	Due to the boundedness of $(\bSigma^\tau, \bu^\tau)$ in $H^1(0,T; S^2\times V)$,
	the second addend goes to zero as $\tau \searrow 0$.
	For the first addend we can proceed as in the proof of \autoref{thm:weak_continuity_solution_operator}:
	the application of \autoref{lem:limsup} yields
	\begin{equation*}
		\limsup_{\tau \searrow 0}
		\int_{0}^{T} \dual{A \dot\bSigma^\tau + B^\star \dot\bu^\tau}{\bT - \bSigma^\tau} \, \d t
		\le
		\int_{0}^{T} \dual{A \dot{\widetilde\bSigma}      + B^\star \dot{\widetilde\bu}     }{\bT - \widetilde\bSigma     } \, \d t.
	\end{equation*}
	Together with \eqref{eq:in_thm:weak_convergence_of_time_discretization} this implies
	\begin{equation*}
		0 \le
		\int_{0}^{T} \dual{A \dot{\widetilde\bSigma}      + B^\star \dot{\widetilde\bu}     }{\bT - \widetilde\bSigma     } \, \d t
		\quad\text{for all } \bT \in L^2(0,T; S^2) \text{, } \bT \in \KK \text{ a.e.\ in } (0,T).
	\end{equation*}
	Therefore, $(\widetilde\bSigma, \widetilde\bu)$ is the solution of \eqref{eq:Lower-Level_Problem}.
	This implies $(\widetilde\bSigma, \widetilde\bu) = (\bSigma, \bu)$.
\end{proof}

\subsection{The time-discrete optimal control problem}
\label{subsec:approx_by_time_discrete}
In this section we consider a time discretization of \eqref{eq:continuous_ulp}.
The time-discrete controls belong to the space $U^N$, which is identified with a subspace of $H_{\{0\}}^1(0,T;U)$ via the piecewise linear interpolation \eqref{eq:linear_interpolation}.
The discretization of the admissible set is given by $\Uadtau = \Uad \cap U^N$, where $\tau = T/N$.
Now, the time-discrete optimal control problem reads
\begin{equation}
	\label{eq:time-discrete_ulp}
	\tag{$\mathbf{P}^\tau$}
	\left.
		\begin{aligned}
			\text{Minimize}\quad & F(\bu^\tau,\bg^\tau) = \psi( \bu^\tau ) + \frac{\nu}{2} \norm{\bg^\tau}_{H^1(0,T;U)}^2 \\
			\text{such that}\quad & (\bSigma^\tau, \bu^\tau) = \GG^\tau( E \bg^\tau) \\ 
			\text{and}\quad & \bg^\tau \in \Uadtau.
		\end{aligned}
		\quad
	\right\}
\end{equation}
For the analysis of \eqref{eq:time-discrete_ulp} we need the additional assumption
that all $\bg \in \Uad$ can be approximated by time-discrete controls $\bg^\tau \in \Uadtau$.
\begin{assumption}
	\label{asm:approximability_of_Uad}
	In addition to \autoref{asm:psi_lsc}
	we suppose that for all $\bg \in \Uad$, there exists $\bg^\tau \in \Uadtau$, such that $\bg^\tau \to \bg$ in $H^1(0,T;U)$ as $\tau \searrow 0$.
\end{assumption}
We remark that this condition is satisfied for both examples of $\Uad$ given after \autoref{asm:psi_lsc}.

The aim of this section is to answer the question whether optimal controls of \eqref{eq:continuous_ulp} can be approximated by optimal controls of \eqref{eq:time-discrete_ulp}.
Similar results for a regularization of an optimal control problem of static plasticity have been proved in \cite[Section~3.2]{HerzogMeyerWachsmuth2009:2}.
These results are very common for the approximation of minimizers of optimal control problems and the technique of proof is applicable to various problems, see also \cite{CasasTroeltzsch2002:1,Barbu1981:1}.
To be precise, we prove that
\begin{itemize}
	\item one global optimum of \eqref{eq:continuous_ulp} can be approximated by global solutions of \eqref{eq:time-discrete_ulp}, see \autoref{thm:continuous_approximation_with_global_solutions},
	\item every strict local optimum of \eqref{eq:continuous_ulp} can be approximated by local solutions of \eqref{eq:time-discrete_ulp}, see \autoref{thm:continuous_approximation_with_local_solutions},
	\item every local optimum of \eqref{eq:continuous_ulp} can be approximated by local solutions of a perturbed problem \hyperref[eq:time-discrete_ulp_mod]{\textup{($\mathbf{P}_\bg^\tau$)}}, see \autoref{thm:continuous_approximation_with_local_solutions_mod}.
\end{itemize}
These results are very important for the analysis in \cite{Wachsmuth2011:4}.
Without these results at hand, one cannot show the necessity of the optimality conditions therein.

Although
both of $\bu^\tau$ and $\bg^\tau$ are optimization variables in \eqref{eq:time-discrete_ulp},
for simplicity
we refer solely to $\bg^\tau$ being (locally or globally) optimal.
This is justified in view of the continuous solution map $\GG^\tau$ of \eqref{eq:Lower-Level_Problem_semidiscretized}.

First we check that there exists a minimizer of the time-discrete optimal control problem \eqref{eq:time-discrete_ulp}.
\begin{lemma}
	\label{lem:p_tau_has_minimizer}
	There exists a global minimizer of \eqref{eq:time-discrete_ulp}.
\end{lemma}
\begin{proof}
	In \citeautoref{Wachsmuth2011:2:subsec:deriv_of_reg} we show in a more general framework that the solution operator $\GG^\tau$ of \eqref{eq:Lower-Level_Problem_semidiscretized} is Lipschitz continuous from $(V')^N$ to $(S^2\times V)^N$.
	Since the control operator $E$ from $U^N$ into $(V')^N$ is compact, $\GG^\tau$ is compact from $U^N$ to $(S^2\times V)^N$.
	Now we proceed similarly as in the proof of \autoref{thm:existence_continuous}.
\end{proof}

Now, we are going to prove the approximation properties.
\begin{lemma}
	\label{thm:continuous_approximation_with_global_solutions}
	Suppose that \autoref{asm:approximability_of_Uad} is fulfilled.
	Let $\{\tau\}$ be a sequence tending to $0$ and let $\bg^\tau$ denote a global solution
	to \eqref{eq:time-discrete_ulp}.
	\begin{enumerate}
		\item 
			Then there exists an accumulation point $\bg$ of $\{\bg^\tau\}$ in $H^1(0,T;U)$ and
		\item
			every weak accumulation point of $\{\bg^\tau\}$ in $H^1(0,T;U)$ is in fact a strong accumulation point in $H^1(0,T;U)$ and a global solution of \eqref{eq:continuous_ulp}.
	\end{enumerate}
\end{lemma}
\begin{proof}
	By \autoref{asm:psi_lsc}, we can choose $\bg_0 \in \Uad$.
	By \autoref{asm:approximability_of_Uad}, there exists a sequence $\bg^\tau_0 \in \Uadtau$ which converges towards $\bg_0$ in $H_{\{0\}}^1(0,T;U)$.
	Hence, by \eqref{eq:a-priori_time} the corresponding displacements $\bu^\tau_0$ converge towards $\bu_0$ in $H_{\{0\}}^1(0,T; V)$.
	Together with the continuity of $\psi$, this implies the convergence of $F(\bu^\tau_0, \bg^\tau_0)$.
	Since $\bg^\tau$ is a global optimum of \eqref{eq:time-discrete_ulp}, we have
	$F(\bu^\tau, \bg^\tau)  \le F(\bu^\tau_0, \bg^\tau_0)$.
	Hence, $\{\bg^\tau\}$ is bounded in $H^1(0,T;U)$.
	This implies the existence of a weakly convergent subsequence in this space.
	Therefore, assertion~(1) follows by assertion~(2).

	To prove assertion~(2), let $\{\bg^\tau\}$ converge weakly towards $\bg$ in $H^1(0,T;U)$ as $\tau \searrow 0$.
	We denote by $(\bSigma^\tau, \bu^\tau) = \GG^\tau(E\bg^\tau)$ the (time-discrete) solution to \eqref{eq:Lower-Level_Problem_semidiscretized} and by $(\bSigma, \bu) = \GG(E\bg)$ the solution to \eqref{eq:Lower-Level_Problem}.
	Due to the weak convergence proven in \autoref{thm:weak_convergence_of_time_discretization}, we have $\bu^\tau \weakly \bu$ in $H^1(0, T; V)$.

	Let $\tilde\bg \in \Uad$ with corresponding displacement $\tilde\bu$ be arbitrary.
	By \autoref{asm:approximability_of_Uad}, there is $\tilde\bg^\tau \in \Uadtau$, such that  $\tilde\bg^\tau \to \tilde\bg$ in $H^1(0, T; L^2(\Gamma_N; \R^d))$.
	We denote the corresponding displacements by $\tilde\bu^\tau$. By \autoref{thm:strong_convergence} we infer $\tilde\bu^\tau \to \tilde\bu$.
	We have
	\begin{align*}
		F( \bu, \bg )
		& \le \liminf F( \bu^\tau, \bg^\tau )
		&& \text{by weak lower semicontinuity of $F$}
		\\
		& \le \liminf F( \tilde\bu^\tau, \tilde\bg^\tau )
		&& \text{by global optimality of $(\bu^\tau, \bg^\tau)$}
		\\
		& = F( \tilde\bu, \tilde\bg ).
		&& \text{by convergence of $(\tilde\bu^\tau, \tilde\bg^\tau)$}
	\end{align*}
	This shows that $\bg$ is a global optimal solution.
	Inserting $\tilde\bg = \bg$ yields the convergence of norms and hence the strong convergence of $\bg^\tau$ in $H^1(0,T; U)$.
\end{proof}

\begin{lemma}
	\label{thm:continuous_approximation_with_local_solutions}
	Suppose \autoref{asm:approximability_of_Uad} is fulfilled.
	Let $\bg$ be a strict local optimum of \eqref{eq:continuous_ulp} w.r.t.\ the topology of $H^1(0,T;U)$.
	Then, for every sequence $\{\tau\}$ tending to $0$, there is a sequence $\{\bg^\tau\}$ of local solutions to \hyperref[eq:time-discrete_ulp]{\textup{($\mathbf{P}^\tau$)}}, such that $\bg^\tau \to \bg$ strongly in $H^1(0,T;U)$ as $\tau \searrow 0$.
\end{lemma}
\begin{proof}
	\textit{Step~(1):}
	Let $\varepsilon > 0$, such that $\bg$ is the unique global optimum in the closed ball $B_\varepsilon(\bg)$ with radius $\varepsilon$ centered at $\bg$.
	We define $\hatUad = \Uad \cap B_\varepsilon(\bg)$.

	\textit{Step~(2):}
	We check that \autoref{asm:approximability_of_Uad} is fulfilled for $\hatUad$.
	Clearly we have $\bg \in \hatUad$.
	Let $\tilde\bg \in \hatUad$ be arbitrary.
	We construct a sequence $\{\tilde\bg^\tau\} \subset \hatUad$, $\tilde\bg^\tau \in U^N$, such that $\tilde\bg^\tau \to \tilde\bg$.
	Since $\Uad$ satisfies \autoref{asm:approximability_of_Uad}, there are sequences of time-discrete functions $(\bg^\tau), (\hat\bg^\tau) \subset \Uad$, converging towards $\bg$ and $\tilde\bg$, respectively.

	If $\norm{\tilde\bg - \bg}_{H^1(0,T;U)} < \varepsilon$, the convergence $\hat\bg^\tau \to \tilde\bg$ implies $\hat\bg^\tau \in B_\varepsilon(\bg)$, and hence $\hat\bg^\tau \in \hatUad \cap \Uadtau$ for sufficiently small $\tau$.

	Otherwise, if $\norm{\tilde\bg-\bg}_{H^1(0,T;U)} = \varepsilon$, we obtain $\lim \norm{\bg - \hat\bg^\tau}_{H^1(0,T;U)} = \varepsilon$.
	Therefore, there exists $\tau_0$, such that for all $\tau \le \tau_0$, we have $\norm{\bg - \bg^\tau}_{H^1(0,T;U)} \le \varepsilon/2$ and $\norm{\bg - \hat\bg^\tau}_{H^1(0,T;U)} > \varepsilon/2$.
	Now, we construct a convex combination $\tilde\bg^\tau$ of $\bg^\tau$ and $\hat\bg^\tau$ which belongs to $B_\varepsilon(\bg)$ and approximates $\tilde\bg$.
	We define the sequence $\tilde\bg^\tau$ by
	\begin{equation*}
		\tilde\bg^\tau = (1-\eta^\tau) \, \hat\bg^\tau + \eta^\tau \, \bg^\tau,
		\quad\text{with }
		\eta^\tau = \max\left(0, \frac{\norm{\bg - \hat\bg^\tau}_{H^1(0,T;U)} - \varepsilon}{\norm{\bg - \hat\bg^\tau}_{H^1(0,T;U)} - \varepsilon/2} \right) \in [0,1].
	\end{equation*}
	In case $\norm{\bg - \hat\bg^\tau}_{H^1(0,T;U)} > \varepsilon$, we find (all norms are those of $H^1(0,T;U)$)
	\begin{align*}
		\norm{\bg - \tilde\bg^\tau}
		& \le (1 - \eta^\tau) \, \norm{\bg - \hat\bg^\tau} + \eta^\tau \, \norm{\bg - \bg^\tau} \\
		& = \frac{1}{\norm{\bg - \hat\bg^\tau} - \varepsilon/2} \, \Bigh(){ \frac\varepsilon2 \, \norm{\bg - \hat\bg^\tau} + (\norm{\bg - \hat\bg^\tau} - \varepsilon) \, \norm{\bg - \bg^\tau} } \\
		& \le \frac{1}{\norm{\bg - \hat\bg^\tau} - \varepsilon/2} \, \Bigh(){ \frac\varepsilon2 \, \norm{\bg - \hat\bg^\tau} + (\norm{\bg - \hat\bg^\tau} - \varepsilon) \, \frac\varepsilon2 } \\
		& = \varepsilon.
	\end{align*}
	This shows $\norm{\bg - \tilde\bg^\tau}_{H^1(0,T;U)} \le \varepsilon$.
	Moreover, $\eta^\tau \in [0,1]$ implies $\tilde\bg^\tau \in \Uad$.
	Therefore, $\tilde\bg^\tau \in \hatUad$.
	Further, $\lim \norm{\bg - \hat\bg^\tau}_{H^1(0,T;U)} = \varepsilon$ implies $\eta^\tau \to 0$ and hence $\tilde\bg^\tau \to \tilde\bg$.
	This shows that $\hatUad$ satisfies \autoref{asm:approximability_of_Uad}.

	\textit{Step~(3):}
	We define the auxiliary problem
	\begin{equation}
		\label{eq:ulp_aux}
		\tag{$\mathbf{P}_{\bg,\varepsilon}$}
		\left.
			\begin{aligned}
				\text{Minimize}\quad & F(\bu,\bg) = \psi( \bu ) + \frac{\nu}{2} \norm{\bg}_{H^1(0,T;U)}^2 \\
				\text{such that}\quad & (\bSigma,\bu) = \GG(E \bg) \\
				\text{and}\quad & \bg \in \Uad \cap B_\varepsilon(\bg).
			\end{aligned}
			\quad
		\right\}
	\end{equation}
	Since $\bg$ is the unique minimum in $\hatUad = \Uad \cap B_\varepsilon(\bg)$, it is the unique global minimizer of \eqref{eq:ulp_aux}.
	Invoking \autoref{thm:continuous_approximation_with_global_solutions} yields the existence of a sequence $\bg^\tau$ of global solutions of the associated time-discrete problem $\textup{($\mathbf{P}_\varepsilon^\tau$)}$, such that $\bg^\tau \to \bg$ in $H^1(0,T;U)$.

	\textit{Step~(4):}
	It remains to check that $\bg^\tau$ are local solutions of \hyperref[eq:time-discrete_ulp]{\textup{($\mathbf{P}^\tau$)}}.
	The convergence $\bg^\tau \to \bg$ in $H^1(0,T;U)$ implies that there is a $\tau_0$ such that $\norm{\bg - \bg^\tau}_{H^1(0,T;U)} \le \varepsilon/2$ holds for all $\tau \le \tau_0$.
	Let $\tau \le \tau_0$ and $\tilde\bg^\tau \in B_{\varepsilon/2}(\bg^\tau) \cap \Uadtau$ be arbitrary.
	By the triangle inequality we infer $\tilde\bg^\tau \in B_\varepsilon(\bg)$. This implies $\tilde\bg^\tau \in \hatUad \cap \Uadtau$. The global optimality of $\bg^\tau$ implies $F(\GG^{\tau, \bu}(\bg^\tau), \bg^\tau) \le F(\GG^{\tau, \bu}(\tilde\bg^\tau), \tilde\bg^\tau)$.
	Hence, $\bg^\tau$ is a local optimum of \hyperref[eq:time-discrete_ulp]{\textup{($\mathbf{P}^\tau$)}} in the neighborhood $B_{\varepsilon/2}(\bg^\tau)$.
\end{proof}

Finally, we address the approximability of a local minimum, which is not assumed to be strict.
Let $\bg$ be a local optimum of \eqref{eq:continuous_ulp} w.r.t.\ the topology of $H^1(0,T;U)$.
We define the modified problem, see also \cite{CasasTroeltzsch2002:1,Barbu1981:1},
\begin{equation}
	\label{eq:time-discrete_ulp_mod}
	\tag{$\mathbf{P}_{\bg}$}
	\left.
		\begin{aligned}
			\text{Minimize}\quad & F_\bg(\bu,\tilde\bg) = \psi( \bu ) + \frac{\nu}{2} \norm{\tilde\bg}_{H^1(0,T;U)}^2 + \frac12 \norm{\tilde\bg - \bg}_{H^1(0,T;U)}^2 \\
			\text{such that}\quad & (\bSigma, \bu) = \GG(E\tilde\bg) \\
			\text{and}\quad & \tilde\bg \in \Uad
		\end{aligned}
	\right\}
\end{equation}
Clearly, $\bg$ becomes a \emph{strict} local optimum of \eqref{eq:time-discrete_ulp_mod}. Analogously to \eqref{eq:time-discrete_ulp} we define the time-discrete approximation \textup{($\mathbf{P}_{\bg}^{\tau}$)}.
\begin{theorem}
	\label{thm:continuous_approximation_with_local_solutions_mod}
	Suppose \autoref{asm:approximability_of_Uad}  is fulfilled.
	Let $\bg$ be a local optimum of \eqref{eq:continuous_ulp} w.r.t.\ the topology of $H^1(0,T;U)$.
	Then, for every sequence $\{\tau\}$ tending to $0$, there is a sequence $(\bg^\tau)$ of local solutions to \textup{($\mathbf{P}_{\bg}^\tau$)}, such that $\bg^\tau \to \bg$ strongly in $H^1(0,T;U)$ as $\tau \searrow 0$.
\end{theorem}
\begin{proof}
	Since the additional term $\norm{\tilde\bg - \bg}_{H^1(0,T;U)}^2$ is weakly lower semicontinuous, this follows analogously to \autoref{thm:continuous_approximation_with_local_solutions}.
\end{proof}
Due to this theorem, we are able to derive \emph{necessary} optimality conditions for \eqref{eq:continuous_ulp}
by passing to the limit with the optimality conditions of \hyperref[eq:time-discrete_ulp_mod]{\textup{($\mathbf{P}_{\bg}^{\tau}$)}}, see \citeautoref{Wachsmuth2011:4:sec:weak_stationarity_quasistatic}.

\appendix
\section{A convergence result in Bochner-Sobolev spaces}
\label{appendix:bochner}
For the proof of \autoref{thm:strong_convergence} we need the following result.
\short{
	It is a consequence of \cite[Theorems~8.5, 8.7]{Krejci1998}, see also the appendix of the extended version of this paper \cite{Wachsmuth2011:1_extended}.
}{It is a consequence of \cite[Theorems~8.5, 8.7]{Krejci1998} which are cited as \hyperref[thm:convergence_in_L^1]{Theorems~\ref*{thm:convergence_in_L^1}} and \ref{thm:lebesgue_dominated_convergence} below.}
\begin{lemma}
	\label{cor:convergence_in_W^1,p}
	Let $X$ be a Hilbert space and $p \in [1,\infty)$ be given. Let $\{u_n\} \subset W^{1,p}(0, T; X)$, $\{g_n\} \subset L^p(0, T; \R)$ be given sequences for $n \in \N \cup \{0\}$ such that
	\begin{subequations}
		\label{eq:condition_of_cor:convergence_in_W^1,p}
		\begin{align}
			u_n &\mrep{{}\to u_0}{{}\le g_n(t)} \text{ in } L^\infty(0, T; X)
			\label{eq:condition_of_cor:convergence_in_W^1,p_1} \\
			g_n &\mrep{{}\to g_0}{{}\le g_n(t)} \text{ in } L^p(0, T; \R),
			\label{eq:condition_of_cor:convergence_in_W^1,p_2} \\
			\norm{ \dot u_n(t) }_X &\le g_n(t) \text{ a.e.\ in } (0,T), \text{ for all } n \in \N, 
			\label{eq:condition_of_cor:convergence_in_W^1,p_3} \\
			\norm{ \dot u_0(t) }_X &= \mrep{g_0(t)}{g_n(t)} \text{ a.e.\ in } (0,T).
			\label{eq:condition_of_cor:convergence_in_W^1,p_4}
		\end{align}
	\end{subequations}
	Then $u_n \to u_0$ in $W^{1,p}(0, T; X)$.
\end{lemma}

\short{}{
	In order to prove \autoref{cor:convergence_in_W^1,p}, we show two auxiliary results.
	\begin{lemma}
		\label{lem:closure_of_weak_convergence}
		Let $X$ be a normed linear space.
		The sequence $\{x_n\}\subset X$ is assumed to be bounded in $X$ and let $V \subset X'$ be some subset of the dual space $X'$.
		If for some $x \in X$
		\begin{equation}
			\label{eq:weak_convergence}
			f(x_n) \to f(x) \quad \text{as } n \to \infty
		\end{equation}
		holds for all $f \in V$, then \eqref{eq:weak_convergence} holds for all $f \in \closure V$, where the closure is taken with respect to the $X'$-norm.
	\end{lemma}
	\begin{proof}
		Let $f \in \closure V$ be given. Then there exists $\{ f_n \} \subset V$ with $f_n \to f$.
		Due to the boundedness of $\{x_n\}$, there is a $C > 0$ with $\norm{x_n}_{X} \le C$ and $\norm{x}_{X} \le C$.
		Let $\varepsilon > 0$ be given. Then there is a $m \in \N$ such that $\norm{f_m - f}_{X'} \le \varepsilon / ( 3 \, C )$.
		Since \eqref{eq:weak_convergence} holds for $f_m$, there exists $n \in \N$ such that $\abs{ f_m(x_n) - f_m(x) } \le \varepsilon / 3$.
		This implies
		\begin{align*}
			\abs{f(x_n) - f(x)}
			& \le \abs{f(x_n) - f_m(x_n)} + \abs{f_m(x_n)-f_m(x)} + \abs{f_m(x)-f(x)} \\
			& \le \norm{f-f_m}_{X'}\,\norm{x_n}_X + \abs{f_m(x_n)-f_m(x)} + \norm{f-f_m}_{X'}\,\norm{x}_X
			\le \varepsilon.
		\end{align*}
		This proves that \eqref{eq:weak_convergence} holds for $f$.
	\end{proof}

	\begin{lemma}
		\label{lem:weak_star_convergence_in_c_star}
		Let $X$ be some Hilbert space and let $\{u_n\} \subset W^{1,1}(0, T; X)$ and $u \in W^{1,1}(0, T; X)$ be given.
		If $u_n \to u$ in $L^\infty(0, T; X)$ and if $\{u_n\}$ is bounded in $W^{1,1}(0, T; X)$, then
		\begin{equation}
			\label{eq:weak_star_convergence_in_c_star}
			\int_0^T \dual{ \dot u_n(t) }{ \varphi(t) } \, \d t \to 
			\int_0^T \dual{ \dot u  (t) }{ \varphi(t) } \, \d t
			\quad\text{for all } \varphi \in C(0, T; X),
		\end{equation}
		i.e.\ the sequence $\{\dot u_n\}$ converges in $C(0, T; X)'$ with respect to the weak-$*$-topology.
	\end{lemma}
	\begin{proof}
		Let us prove that \eqref{eq:weak_star_convergence_in_c_star} holds for all $\varphi \in C^1(0, T; X)$.
		Due to $W^{1,1}(0, T; X) \subset C(0, T; X)$, integrating by parts (see \cite[(8.24)]{Krejci1998}) and the convergence of $u_n$ in $L^\infty(0, T; X)$ yields
		\begin{align*}
			\int_0^T \dual{ \dot u_n(t) }{      \varphi(t) } \, \d t
			& =  -\int_0^T \dual{      u_n(t) }{ \dot \varphi(t) } \, \d t + \dual{ u_n(T) }{ \varphi(T) } - \dual{ u_n(0) }{ \varphi(0) } \\
			&\to -\int_0^T \dual{      u  (t) }{ \dot \varphi(t) } \, \d t + \dual{ u  (T) }{ \varphi(T) } - \dual{ u  (0) }{ \varphi(0) } \\
			& =   \int_0^T \dual{ \dot u  (t) }{      \varphi(t) } \, \d t
	\end{align*}
	for all $\varphi \in C^1(0, T; X)$.
	Using $\cl C^1(0, T; X) = C(0, T; X)$ with respect to the $\norm{\cdot}_{L^\infty(0, T; X)}$-norm, \autoref{lem:closure_of_weak_convergence} yields the claim.
\end{proof}
We remark that \autoref{lem:weak_star_convergence_in_c_star} follows directly from \cite[Theorem~8.16]{Krejci1998}, with the settings
\begin{align*}
	u_n := u := \varphi, \quad
	\xi_n := v_n, \quad
	\xi := v,
\end{align*}
since $\operatorname{Var}_{[0,T]} v_n \le c$ holds due to the boundedness of $v_n$ in $W^{1,1}(0, T; X)$.
The direct proof is presented for convenience of the reader.

\begin{theorem}[\protect{\cite[Theorem~8.7]{Krejci1998}}]
	\label{thm:convergence_in_L^1}
	Let $X$ be a Hilbert space and $\{ v_n \} \subset L^1(0, T; X)$, $\{ g_n \} \subset L^1(0, T; \R)$ be given sequences for $n \in \N \cup \{0\}$ such that
	\begin{subequations}
		\label{eq:condition_of_thm:convergence_in_L^1}
		\begin{align}
			\vphantom{a^{a^{a^{a^{a^{a}}}}}}\smash{\lim_{n\to\infty} \int_0^T} \dual{ v_n(t) - v_0(t) }{ \varphi(t)} \, \d t &= \mrep{0}{g_n(t)} \text{ for all } \varphi \in C(0, T; X),
			\label{eq:condition_of_thm:convergence_in_L^1_1} \\
			g_n &\mrep{{}\to g_0}{{}\le g_n(t)} \text{ in } L^1(0, T; \R),
			\label{eq:condition_of_thm:convergence_in_L^1_2} \\
			\norm{ v_n(t) }_X &\le g_n(t) \text{ a.e.\ in } (0,T), \text{ for all } n \in \N,
			\label{eq:condition_of_thm:convergence_in_L^1_3} \\
			\norm{ v_0(t) }_X &= \mrep{g_0(t)}{g_n(t)} \text{ a.e.\ in } (0,T).
			\label{eq:condition_of_thm:convergence_in_L^1_4}
		\end{align}
	\end{subequations}
	Then $v_n \to v_0$ in $L^1(0, T; X)$.
\end{theorem}

\begin{theorem}[\protect{\cite[Theorem~8.5]{Krejci1998}}]
	\label{thm:lebesgue_dominated_convergence}
	Let $B$ be a Banach space and $p \in [1,\infty)$ be given. Let $\{v_n\} \subset L^p(0, T; B)$, $\{g_n\} \subset L^p(0, T; \R)$ be given sequences for $n \in \N \cup \{0\}$ such that
	\begin{subequations}
		\label{eq:condition_of_thm:lebesgue_dominated_convergence}
		\begin{align}
			v_n(t) &\to v_0(t) \text{ a.e.\ in } (0,T),
			\label{eq:condition_of_thm:lebesgue_dominated_convergence2} \\
			g_n &\to \mrep{g_0}{v_0(t)} \text{ in } L^p(0, T; \R),
			\label{eq:condition_of_thm:lebesgue_dominated_convergence1} \\
			\norm{ v_n(t) }_B &\mrep{{}\le g_n(t)}{{}\to v_0(t)} \text{ a.e.\ in } (0,T), \text{ for all } n \in \N \cup \{0\},
			\label{eq:condition_of_thm:lebesgue_dominated_convergence3}
		\end{align}
	\end{subequations}
	Then $v_n \to v_0$ in $L^p(0, T; B)$.
\end{theorem}

\begin{corollary}
	\label{cor:convergence_in_L^p}
	Let $X$ be a Hilbert space and $p \in [1,\infty)$ be given. Let $\{v_n\} \subset L^p(0, T; X)$, $\{g_n\} \subset L^p(0, T; \R)$ be given sequences for $n \in \N \cup \{0\}$ such that
	\begin{subequations}
		\label{eq:condition_of_cor:convergence_in_L^p}
		\begin{align}
			\vphantom{a^{a^{a^{a^{a^{a}}}}}}\smash{\lim_{n\to\infty} \int_0^T} \dual{ v_n(t) - v_0(t) }{ \varphi(t)} \, \d t &= \mrep{0}{g_n(t)} \text{ for all } \varphi \in C(0, T; X),
			\label{eq:condition_of_cor:convergence_in_L^p_1} \\
			g_n &\mrep{{}\to g_0}{{}\le g_n(t)} \text{ in } L^p(0, T; \R),
			\label{eq:condition_of_cor:convergence_in_L^p_2} \\
			\norm{ v_n(t) }_X &\le g_n(t) \text{ a.e.\ in } (0,T), \text{ for all } n \in \N,
			\label{eq:condition_of_cor:convergence_in_L^p_3} \\
			\norm{ v_0(t) }_X &= \mrep{g_0(t)}{g_n(t)} \text{ a.e.\ in } (0,T).
			\label{eq:condition_of_cor:convergence_in_L^p_4}
		\end{align}
	\end{subequations}
	Then $v_n \to v_0$ in $L^p(0, T; X)$.
\end{corollary}
\begin{proof}
	Applying \autoref{thm:convergence_in_L^1} yields $v_n \to v$ in $L^1(0, T; X)$.
	Thus there exists a pointwise convergent subsequence $\{ v_{n_k} \}$.
	Hence, \autoref{thm:lebesgue_dominated_convergence} yields $v_{n_k} \to v$ in $L^p(0, T; X)$.
	Since we can start with an arbitrary subsequence of $v_n$, we obtain $v_n \to v$ in $L^p(0, T; X)$.
\end{proof}

\begin{proof}[Proof of \autoref{cor:convergence_in_W^1,p}]
	By \eqref{eq:condition_of_cor:convergence_in_W^1,p_2} and \eqref{eq:condition_of_cor:convergence_in_W^1,p_3} we obtain the boundedness of $u_n$ in $W^{1,1}(0, T; X)$.
	\autoref{lem:weak_star_convergence_in_c_star} yields \eqref{eq:weak_star_convergence_in_c_star} and hence we can apply
	\autoref{cor:convergence_in_L^p} for $v_n = \dot u_n$.
	This yields $\dot u_n \to \dot u_0$ in $L^p(0, T; X)$ and hence $u_n \to u_0$ in $W^{1,p}(0, T; X)$.
\end{proof}
}

\subsection*{Acknowledgment}
\pdfbookmark[1]{Acknowledgement}{toc}
The author would like to express his gratitude to Roland Herzog and Christian Meyer
for helpful discussions on the topic of this paper
and to Dorothee Knees for pointing out reference \cite{Krejci1998}.

\ifSIAM
This work was supported by a DFG grant within the \href{http://www.am.uni-erlangen.de/home/spp1253/wiki}{Priority Program SPP~1253} (\emph{Optimization with Partial Differential Equations}), which is gratefully acknowledged.
\else
This work was supported by a DFG grant within the \href{http://www.am.uni-erlangen.de/home/spp1253/wiki}{Priority Program SPP~1253} (\emph{Optimization with Partial Differential Equations}), which is gratefully acknowledged.

\fi


\pdfbookmark[1]{\bibname}{toc}
\ifSIAM
\bibliographystyle{siam}
\bibliography{Optimal_Control_of_Quasistatic_Plasticity_Part_I}
\else
\bibliographystyle{plainnat}
\bibliography{Plasticity}
\fi

\end{document}